\documentclass[a4paper, 11pt,reqno]{amsart}
\usepackage[top = 1in, bottom = 1in, left=1.2in, right=1.2in]{geometry}

\usepackage[utf8]{inputenc}
\usepackage[activeacute,english]{babel}
\usepackage[T1]{fontenc}
\usepackage{verbatim}
\usepackage{lmodern}
\usepackage{mathtools}
\usepackage{xcolor}
\usepackage{hyperref}
\hypersetup{colorlinks, linkcolor={red!50!black}, citecolor={green!50!black}, urlcolor={blue!50!black}}

\usepackage{microtype}

\usepackage[shortlabels]{enumitem}

\usepackage[linesnumbered,ruled,vlined]{algorithm2e}
\usepackage{amsmath,amssymb,amsfonts}
\usepackage{amsthm}
\usepackage{dsfont}
\usepackage{cite}

\def\NN{\mathbb{N}}

\DeclareMathOperator{\prob}{\mathbb{P}}
\DeclareMathOperator{\expec}{\mathbb{E}}

\newcommand{\whp}{whp}
\newcommand{\Whp}{Whp}
\newcommand{\cG}{\mathcal{G}}
\newcommand{\Po}{\mathrm{Po}}
\newcommand{\cF}{\mathcal{F}}
\newcommand{\cV}{\mathcal{V}}
\newcommand{\cT}{\mathcal{T}}
\newcommand{\cW}{\mathcal{W}}
\newcommand{\cH}{H^r(n,p)}
\newcommand{\EE}{\mathbb{E}}
\newcommand{\eps}{\varepsilon}
\newcommand{\eone}{E_{w}(\ell)}

\newcommand{\Bi}{\mathrm{Bi}}
\newcommand{\Tstop}{T_{\mathrm{stop}}}
\newcommand{\coredetector}{\textnormal{\texttt{CoreConstruct}}}
\newcommand{\coredetectortwo}{\coredetector }

\renewcommand{\Pr}{\mathbb{P}}

\newcommand{\varprop}[1]{\xi_{#1}}
\newcommand{\varpropl}[2]{\xi_{#1}^{(#2)}}
\newcommand{\factprop}[1]{\hat \xi_{#1}}
\newcommand{\factpropl}[2]{\hat \xi_{#1}^{(#2)}}

\newcommand{\varpropp}[1]{\xi_{#1}'}

\theoremstyle{theorem}
\newtheorem{theorem}{\textbf{Theorem}}[section]
\newtheorem{corollary}[theorem]{\textbf{Corollary}}
\newtheorem{lemma}[theorem]{\textbf{Lemma}}
\newtheorem{claim}[theorem]{\textbf{Claim}}
\newtheorem{question-formatless}{\textbf{Question}}
\newtheorem{proposition}[theorem]{\textbf{Proposition}}
\theoremstyle{definition}
\newtheorem{definition}[theorem]{\textbf{Definition}}

\newtheorem{remark}[theorem]{\textbf{Remark}}
\theoremstyle{fact}
\newtheorem{fact}[theorem]{\textbf{Fact}}

\newcommand{\intvarsurvlim}{\rho_*}
\newcommand{\invintvarsurvlim}{\tau_*}
\newcommand{\intfactsurvlim}{\hat \rho_*}

\newcommand{\intvarsurv}[1]{\rho_{#1}}
\newcommand{\intfactsurv}[1]{\hat \rho_{#1}}

\newcommand{\error}{\eps}
\newcommand{\oldeta}{\beta}
\newcommand{\oldbeta}{\eta}
\newcommand{\propnumvs}{\alpha}
\newcommand{\numvs}[1]{v(#1)}
\newcommand{\numvsdeg}[2]{v_{#1}(#2)}
\newcommand{\numedg}[1]{e(#1)}

\newcommand{\randfact}{G^r(n,p)}

\newcommand{\pc}{P_G}

\newcommand{\pcprime}{P^{'}_G}
\newcommand{\paddedcore}{padded core}
\newcommand{\Paddedcore}{Padded core}

\newcommand{\largedegsquare}{\widetilde{\mathcal{D}}}
\newcommand{\gooddegs}{\mathcal{D}}
\newcommand{\coreh}{C_H}
\newcommand{\coreg}{C_G}

\newcommand{\corer}{R_G}

\newcommand{\corerzero}{R_{G_0}}
\newcommand{\barG}{\overline{G}}

\newcommand{\td}{\tilde d}

\newcommand{\hd}{d^*}
\newcommand{\hD}{D^*}

\newcommand\numeq[1]
  {\stackrel{\scriptscriptstyle\mkern-1.5mu#1\mkern-1.5mu}{=}}
\newcommand\numgeq[1]
  {\stackrel{\scriptscriptstyle\mkern-1.5mu#1\mkern-1.5mu}{\geq}}
  \newcommand\numleq[1]
  {\stackrel{\scriptscriptstyle\mkern-1.5mu#1\mkern-1.5mu}{\leq}}

\begin{document}

\title[Loose cores and cycles in random hypergraphs]{Loose cores and cycles in random hypergraphs}

\author[O. Cooley, M. Kang, J. Zalla]{Oliver Cooley$^{*}$,
Mihyun Kang$^{*}$,
Julian Zalla$^{*}$}

\renewcommand{\thefootnote}{\fnsymbol{footnote}}

\footnotetext[1]{Supported by Austrian Science Fund (FWF): I3747, W1230, \texttt{\{cooley,kang,zalla\}@math.tugraz.at}, Institute of Discrete Mathematics, Graz University of Technology, Steyrergasse 30, 8010 Graz, Austria.}

\renewcommand{\thefootnote}{\arabic{footnote}}
\begin{abstract}
Inspired by the study of loose cycles in hypergraphs, we define the \textit{loose core} in hypergraphs as a structure
which mirrors the close relationship between cycles and $2$-cores in graphs.
We prove that in the $r$-uniform binomial random hypergraph $H^r(n,p)$,
the order of the loose core undergoes a phase transition at a certain critical threshold
and determine this order, as well as the number of edges, asymptotically in the subcritical and supercritical regimes. 

Our main tool is an algorithm called \coredetector, which enables us to analyse a peeling process for the loose core. By analysing this algorithm we determine the asymptotic degree distribution of vertices in the loose core and in particular how many vertices and edges the loose core contains. As a corollary we obtain an improved upper bound on the length of the longest loose cycle in $H^r(n,p)$.
\end{abstract}

\maketitle

\section{Introduction} \label{sec:intro}

\subsection{Motivation}

One of the first phase transition results  for random graphs  is the celebrated result of Erd\H{o}s and R\'enyi~\cite{ErdosRenyi60} on the emergence of a \emph{giant component} of linear
order when the number of edges passes $\frac{n}{2}$, or from the viewpoint that is now more common,
when the edge probability passes $\frac{1}{n}$. This result has since been strengthened and generalised
in a number of directions. In particular in hypergraphs it has been extended to vertex-components
(e.g.~\cite{BCOK10,BCOK14,BollobasRiordan12,BollobasRiordan17,KaronskiLuczak02,Poole15,SPS85})
as well as to high-order components (e.g.~\cite{CFDGK20,CKK18,CKK19,CKP18}).

The $k$-core of a graph $G$, defined as the maximal subgraph of minimum degree at least~$k$, has been studied extensively in the literature (e.g.~\cite{Cooper2004, JansonLuczak07, kimcore, LuczakCore1991, PittelSpencerWormald96}). 
In random graphs, the $k$-core may be seen as a natural generalisation of the largest component: 
in the case $k=2$, whp\footnote{short for \emph{with high probability}, i.e.\ with probability tending to one as the number of vertices $n\to \infty$.}
a linear-sized $2$-core emerges at the same time as the giant component, and indeed
lies almost entirely within the giant component, while for $k\ge 3$, whp the $k$-core
is identical to the largest $k$-connected subgraph~\cite{LuczakCore1991,Luczak92}. 
In~\cite{LuczakCore1991} {\L}uczak estimated the order of (i.e.\ the number of vertices in)
the $k$-core of $G(n,p)$ and the asymptotic probability that the $k$-core is $k$-connected.
{\L}uczak also showed in~\cite{Luczak92} that in the random graph process, in which edges are added to an empty graph one by one in
a uniformly random order, \whp\ at the moment the $k$-core first becomes non-empty, its order
is already linear in~$n$. 
A crucial milestone was achieved by Pittel, Spencer and Wormald~\cite{PittelSpencerWormald96},
who for $k\geq 3$ determined the threshold probability at which the non-empty $k$-core appears \whp\ and
determined its asymptotic order and size (i.e.\ number of edges).
This was strengthened by Janson and Luczak~\cite{JansonLuczak08}, who proved a bivariate
central limit theorem for the order and size of the $k$-core.
Cain and Wormald~\cite{CainWormald06} determined the asymptotic distribution
of vertex degrees within the $k$-core.
Further research has focussed for example on the robustness of the core against edge deletion~\cite{Sato14}
and how quickly the peeling process arrives at the core~\cite{AchlioptasMolloy15,Gao18,GaoMolloy18,JMT14}.
There are many more results in the literature for cores in random graphs, see e.g. \cite{Cooper2004, JansonLuczak07, kimcore}.

Paths and cycles in random graphs have been investigated at least since 1979 by de la Vega~\cite{delavega79} and somewhat later by Ajtai, Koml\'os,
and Szemer\'edi~\cite{AjtaiKomlosSzemeredi81}.
Regarding the length of the longest path in the random graph $G(n,p)$,
a standard ``sprinkling'' argument (see Lemma~\ref{lem:standardsprinkling} with $r=2$)
shows that in the supercritical regime the length of the longest path and cycle are very similar.
Thus it follows from the results of {\L}uczak~\cite{Luczak91} on the length of the longest cycle
that for $\eps=\eps(n)=o(1)$ and $p=\frac{1+\eps}{n}$ (i.e.\ shortly after the phase transition),
under the assumption $\eps^5n\to\infty$ \whp\ the longest path has length $\Theta(\eps^2 n)$, where explicit constants can be given.
The best-known upper bounds derive from a careful analysis of the $2$-core
and the simple observation that any cycle must lie within the $2$-core.

There are many different ways of generalising the concept of a $k$-core to hypergraphs; some results for these cores can be found in e.g. Molloy~\cite{Molloy2005} and Kim~\cite{kimcore}. However,
in the case $k=2$,
all $k$-cores which have been studied so far do not fully capture the nice connection between the $2$-core and cycles in graphs.
In~\cite{Molloy2005} Molloy determined the threshold for the appearance of a non-trivial $k$-core
(in that paper defined as a subhypergraph where every vertex has degree at least $k$)
in the $r$-uniform binomial random hypergraph $\cH$ for all $r,k\geq 2$ such that $r+k\geq 5$.
The proof relied on a clever heuristic argument which was first introduced by Pittel, Spencer and Wormald in~\cite{PittelSpencerWormald96} and has been adapted by many other authors, see e.g.~\cite{JansonLuczak07, kimcore,Riordan08,skubch2015core}. It turns out that the proofs in~\cite{Molloy2005} can be extended to a wide range of core-type structures.
In the case $k=2$, Dembo and Montanari~\cite{DemboMontanari08} strengthened this by determining
the width of, and examining the behaviour within, the critical window.

One of the most natural  concepts of paths and cycles in hypergraphs
is \emph{loose paths} and \emph{loose cycles} (see Definition~\ref{def:loosecycle}).
A special case of a recent result of Cooley, Garbe, Hng, Kang, Sanhueza-Matamala and Zalla~\cite{cooley2020longest}
shows that the length of the longest loose path in an $r$-uniform binomial random hypergraph undergoes
a phase transition from logarithmic length to linear, and they also determined the critical threshold,
as well as proving upper and lower bounds on the length in the subcritical and supercritical ranges.

Inspired by the substantial body of research on loose cycles, in this paper we introduce the
\emph{loose core} (see Definition~\ref{def:loosecore}), a structure which does indeed capture the connection between cores and cycles in hypergraphs.
Our first main result concerns the degree distribution of vertices in the loose core (see Theorem~\ref{thm:mainresultdegrees}). In fact we prove a stronger result regarding degree distributions of both vertices and edges (see Theorem~\ref{thm:factor:reducedcoredegs}).
As a consequence we can deduce both the asymptotic numbers of vertices and edges in the loose core (see Theorem~\ref{thm:mainresultorder}) and an improved upper bound on the length of the longest loose cycle in an $r$-uniform binomial random hypergraph  (see Theorem~\ref{thm:mainresultcycle}). 

Before stating our main results, in the next section we introduce some definitions and notations which we will use throughout the paper.

\subsection{Setup}\label{def:basicdefinitions}
Given a natural number $r\geq 3$,
an \emph{$r$-uniform hypergraph} consists of a vertex set $V$ and an edge set
$E\subset\binom{V}{r}$, where $\binom{V}{r}$ denotes the set of all $r$-element subsets of $V$. Let $H^r(n,p)$ denote the \emph{$r$-uniform binomial random hypergraph}
on vertex set $[n]$ in which each set of $r$ distinct vertices forms an edge with probability $p$ independently. For any positive integer $k$ we write $[k]\coloneqq\{1,\ldots,k\}$ and $\ [k]_0\coloneqq\{0,\ldots,k\}$.
We also include $0$ in the natural numbers, so we write $\NN=\{0,1,\ldots\}$ and $\NN_{\geq k}\coloneqq\{k,k+1,\ldots\}$.
Throughout the paper, unless otherwise stated any asymptotics are taken as $n\to\infty$. 
In particular, we use the standard
Landau notations $o(\cdot)$, $O(\cdot),\Theta(\cdot),\omega(\cdot)$
with respect to these asymptotics.

The loose core will be defined in terms of two parameters, namely the standard notion of (vertex-)degree
and a notion we call the connection number.

\begin{definition}
Let $H$ be an $r$-uniform hypergraph. Let $d_H(v)$ be the \emph{degree} of a vertex $v$ in $H$ (i.e.\ the number of edges which contain it) and let $\delta(H)$ denote the \emph{minimum (vertex-)degree} of $H$, i.e.\ the smallest degree of any vertex of $H$. For any edge $e\in E(H)$, define the \textit{connection number} $\kappa(e)\in[r]_0$ of $e$ as \[\kappa(e)=\kappa_H(e)\coloneqq|\{v\in e: d_H(v)\geq 2\}|\] and let $\kappa(H)\coloneqq\min\limits_{e\in E(H)}\kappa(e)$.
\end{definition}
We are now ready to define the loose core.
\begin{definition}[Loose core]\label{def:loosecore}
The \textit{loose core} of an $r$-uniform hypergraph $H$ is the maximal subhypergraph $H'$ of $H$ such that 
\begin{enumerate}[label=\textnormal{\textbf{(C\arabic*)}}]
    \item\label{item:loosecorefirstcond} $\delta(H')\geq 1$, 
    \item\label{item:loosecoresecondcond} $\kappa(H')\geq 2$.
\end{enumerate}
If such a subhypergraph does not exist, then we define the loose core to be the empty hypergraph (i.e.\ the hypergraph with no vertices and no edges).
\end{definition}
Note that the loose core is unique, since the union of two hypergraphs each with properties \ref{item:loosecorefirstcond} and \ref{item:loosecoresecondcond} again has these properties. The first condition in Definition~\ref{def:loosecore} simply states that the loose core contains no isolated vertices and the second condition specifies how edges are connected to each other in the loose core. Note that for $r\ge 3$ the loose core might contain vertices of degree $1$, in contrast to the graph case. For $r=2$, Definition~\ref{def:loosecore} coincides with the $2$-core of a graph.

Our motivation to study loose cores arises from the study of loose cycles in hypergraphs which are closely related to loose paths.
\begin{definition}[Loose path/cycle]\label{def:loosecycle}
A \textit{loose path of length $\ell$} in an $r$-uniform hypergraph is a sequence of distinct vertices $v_1,\ldots,v_{\ell(r-1)+1}$ and a sequence of edges $e_1,\ldots,e_{\ell}$, where $e_i=\{v_{(i-1)(r-1)+1},\ldots,v_{(i-1)(r-1)+r}\}$ for $i\in[\ell]$. A \textit{loose cycle of length $\ell$} in an $r$-uniform hypergraph is defined similarly except that $v_{\ell(r-1)+1}=v_1$ (and otherwise all vertices are distinct).
\end{definition}

Note that for $i \in [\ell-1]$ we have $e_i \cap e_{i+1}=\{v_{i(r-1)+1}\}$ (and in the case of a loose cycle,
$e_\ell \cap e_1 = \{v_1\}$), so in particular two consecutive edges
intersect in precisely one vertex.
Observe that a loose cycle satisfies conditions~\ref{item:loosecorefirstcond} and~\ref{item:loosecoresecondcond} of a loose core (Definition~\ref{def:loosecore}) and hence it must be contained in the maximal subhypergraph with these properties,
i.e.\ in the loose core.

We will now define various parameters which will occur often in this paper. Some of these
definitions may seem arbitrary and unmotivated initially, but their meaning will become clearer over the course of the paper.
Given $d>0$, consider a sequence $(d_n)_{n \in \NN}$ of real numbers such that $d_n\to d$.
Then for
$r\in\NN_{\geq 3}$ and $n\in\NN$, set
\[
p=p(r,n)\coloneqq\frac{d_n}{\binom{n-1}{r-1}}, \qquad d^*=d^*(r)\coloneqq\frac{1}{r-1}.
\]
In addition we define a function $F:[0,\infty) \to\mathbb{R}$ by setting
\begin{equation*}\label{eq:gfunctiondefinition}
    F(x)=F_{r,d}(x)\coloneqq\exp{\left(-d\left(1-x^{r-1}\right)\right)}
\end{equation*}
and let $\intvarsurvlim=\intvarsurvlim(r,d)$ be the largest solution\footnote{$\intvarsurvlim$ is well-defined since $0$ is certainly a solution and the set of solutions is closed by continuity.} of the fixed-point equation
\begin{equation}\label{eq:fixedpointequation}
    1-\rho=F(1-\rho).  
\end{equation}
Since the function $F$ is dependent on $d$, so too are the solutions to this equation.
It turns out that $d^*$ is a threshold at which the solution set changes its behaviour from only containing the trivial solution $0$ to containing a unique positive solution (see Claim~\ref{claim:behaviouroffixedpointsol}).

We define
\begin{equation}\label{eq:relatingtherhos}
\intfactsurvlim=\intfactsurvlim(r,d)\coloneqq1-(1-\intvarsurvlim)^{r-1}
\end{equation}
and
\begin{equation}\label{eq:eta}
\oldbeta=\oldbeta(r,d)\coloneqq1-\frac{(r-1)\intvarsurvlim(1-\intvarsurvlim)^{r-2}}{\intfactsurvlim}.
\end{equation}
 Furthermore let 
\begin{equation*}\label{eq:alphadefinition}
    \propnumvs=\propnumvs(r,d)\coloneqq\intvarsurvlim\left(1-d(r-1)(1-\intvarsurvlim)^{r-1}\right),
\end{equation*}
\begin{equation}\label{eq:definitionofeta}
\oldeta=\oldeta(r,d)\coloneqq\frac{d}{r}\left(1-(1-\intvarsurvlim)^ r-r\intvarsurvlim(1-\intvarsurvlim)^ {r-1}\right),
\end{equation}
and
\begin{equation}\label{eq:definitionofgamma}
    \gamma=\gamma(r,d)\coloneqq1-\exp(-d\intfactsurvlim)-d\intfactsurvlim \exp(-d\intfactsurvlim).
\end{equation}

\subsection{Main results: loose cores and cycles in hypergraphs}\label{sec:mainresults}

For any $j\in\NN_{\ge 1}$, let $\numvsdeg{j}{\coreh}$ be the number of vertices of $H= H^r(n,p)$
with degree~$j$ in the loose core $\coreh$ of $H$
and let 
\[\mu_j\coloneqq\numvsdeg{j}{\coreh}\cdot n^{-1}.\]
Let $\numvs{\coreh} = \sum_{j \ge 1}\numvsdeg{j}{\coreh}$ denote the number of vertices and $\numedg{\coreh}$ the number of edges in the loose core  $\coreh$ in $H$.
We also define $\numvsdeg{0}{\coreh}$ to be the number of vertices of $H$ which are not in the loose core of $H$ (so $v_0(\coreh)=n-v(\coreh)$),
and $\mu_0\coloneqq \numvsdeg{0}{\coreh}\cdot n^{-1}$. (Observe that this notation is consistent if, with a slight abuse of terminology,
we view vertices which are not in the loose core as having degree $0$ in the loose core.) We interpret a $\Po(0)$ variable as being deterministically $0$. 

Our first main result describes the asymptotic degree distribution of vertices in the loose core $\coreh$  of $H= H^r(n,p)$.
\begin{theorem}\label{thm:mainresultdegrees}
Let $r,d,p,\intfactsurvlim$ and $\oldbeta$  be as in Section~\ref{def:basicdefinitions} and  let $H=H^r(n,p)$. Let $Y$ be a random variable with distribution $\Po(d\intfactsurvlim)$ and define
\[
    Z \coloneqq \bigg\{\begin{array}{lr}
        Y & \text{if } Y\neq 1,\\
        \mathrm{Ber}\left(\oldbeta\right) & \text{if } Y=1.
        \end{array}\] 
Then there exists $\eps = \eps(n) = o(1)$ such that \whp\ for any constant $j\in\NN$ we have
\[\mu_j=\prob(Z=j)\pm \eps.\]
\end{theorem}
Our second main result describes the asymptotic numbers of vertices and edges in the loose core  $\coreh$  of $H= H^r(n,p)$.

\begin{theorem}\label{thm:mainresultorder}
Let $r,p, \alpha$ and $\beta$  be as in Section~\ref{def:basicdefinitions} and  let $H=H^r(n,p)$.
Then \whp
		\[\numvs{\coreh}=(\propnumvs+o(1)) n\]
		and
		\[\numedg{\coreh}=(\oldeta+o(1))n.\]
\end{theorem}

By analysing the loose core we obtain an upper bound on the length of the longest loose cycle in $\cH$.

\begin{theorem}\label{thm:mainresultcycle} 
Let $r,p, \beta$ and $\gamma$  be as in Section~\ref{def:basicdefinitions} and  let $H=H^r(n,p)$.
Let $L_C$ be the length of the longest loose cycle in $H$. Then \whp
\[L_C\leq \left(\min\{\oldeta,\gamma\}+o(1)\right)\cdot n.\]
\end{theorem}

In fact, a standard ``sprinkling'' argument shows that whp the longest loose path in $\cH$ is not significantly longer than the longest loose cycle and therefore we obtain the following corollary.

\begin{corollary}\label{cor:upperboundlongestpath} 
Let $r,p, \beta$ and $\gamma$  be as in Section~\ref{def:basicdefinitions} and  let $H=H^r(n,p)$. 
Let $L_P$ be the length of the longest loose path in $H$. Then \whp
\[L_P\leq \left(\min\{\oldeta,\gamma\}+o(1)\right)\cdot n.\]
\end{corollary}

As mentioned previously, Claim~\ref{claim:behaviouroffixedpointsol} will state that $d^*=\frac{1}{r-1}$ is a threshold at which the solution set of~\eqref{eq:fixedpointequation} changes its behaviour from only containing $0$ to containing a unique positive solution. Together with Theorem~\ref{thm:mainresultcycle} and Corollary~\ref{cor:upperboundlongestpath} (and recalling the definitions of $\beta,\gamma$
in~\eqref{eq:definitionofeta} and~\eqref{eq:definitionofgamma}) this implies that $d^*=\frac{1}{r-1}$ is a threshold for the existence of a loose path/cycle of linear order,
and it is interesting in particular to examine the behaviour shortly after the phase transition.
\begin{corollary}\label{cor:shortlyafterphasetransition}
Let $r\in\NN_{\geq 3}$, let $\eps>0$ be constant and let $p=\frac{1+\eps}{(r-1)\binom{n-1}{r-1}}$. Let $L_C$ and $L_P$ be the length of the longest loose cycle and  the longest loose  path in $\cH$.  Then \whp
\[L_C\leq L_P+1\leq \left(\frac{2\eps^2}{(r-1)^2}+O(\eps^3)\right)\cdot n.\]
\end{corollary}

In other words, we have an upper bound on $L_C$ and $L_P$ in the barely supercritical regime.
For a corresponding lower bound, we will quote (a special case of) a result from~\cite{cooley2020longest}
(which we later state formally as Theorem~\ref{thm:pathsresult}) which gives a lower bound on $L_P$.
By applying the sprinkling argument again we also obtain a lower bound on $L_C$,
and together with Corollary~\ref{cor:shortlyafterphasetransition} we obtain the following.

\begin{theorem}\label{thm:bestknownresultcycles}
Let $r\in\NN_{\geq 3}$, let $\eps>0$ be constant and let $p=\frac{1+\eps}{(r-1)\binom{n-1}{r-1}}$.  Let $L_C$ and $L_P$ be the length of the longest loose cycle and  the longest loose  path in $\cH$.   Then \whp\ 
\[
\left(\frac{\eps^2 }{4(r-1)^2}+O(\eps^3)\right)\cdot n\leq L_C\leq L_P+1\leq\left(\frac{2\eps^2}{(r-1)^2}+O(\eps^3)\right)\cdot n .
\]
\end{theorem}
Theorem~\ref{thm:bestknownresultcycles} provides the best known upper and lower bounds on $L_P,L_C$ in the regime when
$p= \frac{1+\eps}{(r-1)\binom{n-1}{r-1}}$, 
but there is a multiplicative factor of~$8$ between these two bounds and the correct asymptotic value is still unknown.
Indeed it is not even clear that the random variables $L_P,L_C$ are concentrated around a single value.

We note that even for cycles in  $G(n,p)$ in the barely supercritical regime the correct asymptotic value is not known
despite considerable efforts in this direction. In particular, the best known bounds
when $0<\eps=\eps(n)=o(1)$ satisfies $\eps^3n\to\infty$ and $p=\frac{1+\eps}{n}$
are due to {\L}uczak~\cite{Luczak91} (lower bound) and Kemkes and Wormald~\cite{KemkesWormald13} (upper bound),
and state that \whp\ the length $L_C$ of the longest cycle satisfies
\[
\left(\frac{4}{3}+o(1)\right)\eps^2n\leq L_C\leq(1.7395+o(1))\eps^2n.
\]

The proofs of all the results of this section appear in Section~\ref{sec:proofofmainresults} as a consequence of a single, more general result (Theorem~\ref{thm:factor:reducedcoredegs}).

\subsection{Key proof techniques}
In order to prove our main results, we transfer the problem from $\cH$ to the \emph{factor graph} $G\coloneqq G(\cH)$
which will be formally defined in Section~\ref{sec:factorgraphs}. In the factor graph we define the \emph{reduced core} $R_G$,
which is closely related to the $2$-core of $G$ and from which we can reconstruct the loose core of $\cH$, but which is easier to analyse. We use a peeling process (Definition~\ref{def:peeling}) and an auxiliary algorithm called $\coredetectortwo$ to determine the asymptotic proportion of variable and factor nodes of $G$ with fixed degree in the reduced core (Theorem~\ref{thm:factor:reducedcoredegs}). We also need martingale techniques, in particular an Azuma-Hoeffding inequality and an associated vertex-exposure martingale to show concentration of the numbers of vertices and edges of fixed degree around the respective expectations. 
\subsection{Paper overview}
The rest of the paper is structured as follows. 

In Section~\ref{sec:prelim} we set basic notation and state some standard probabilistic lemmas which we will use later.
In Section~\ref{sec:factorgraphs} we switch our focus to factor graphs,
define the reduced core and state Theorem~\ref{thm:factor:reducedcoredegs}
which describes degree distributions in the reduced core and which
implies all of our main results, as we prove in Section~\ref{sec:proofofmainresults}.

Subsequently, Section~\ref{sec:peeling} describes a standard peeling process to obtain the reduced core and contains two main lemmas
which together imply Theorem~\ref{thm:factor:reducedcoredegs}.
The first of these (Lemma~\ref{lem:mainlemma1}) describes the degree distribution after a sufficiently large number of steps of the peeling process,
and will be proved in Section~\ref{sec:algorithm}.
The second main lemma (Lemma~\ref{lem:mainlemma2}) states that subsequently, very few further vertices will be deleted in the remainder of the peeling process,
and therefore this degree distribution is also a good approximation for the degree distribution in the reduced core.
Lemma~\ref{lem:mainlemma2} will be proved in Section~\ref{sec:prooflowerbound}.

In Section~\ref{sec:concluding}, we conclude with some discussion and open questions.
We omit from the main body of the paper many proofs which simply involve technical calculations
or standard applications of common methods,
but include them in the appendices for completeness.
Appendix~\ref{app:analysisfixedpoint} contains an analysis of the fixed-point equation~\eqref{eq:fixedpointequation},
while Appendix~\ref{app:problemmas} contains the proofs of some basic probabilistic lemmas which are needed throughout the paper.
Finally, Appendix~\ref{app:eventE} and Appendix~\ref{sec:proofofuniformity} constitute the proofs of Lemma~\ref{lem:eventE}
and Lemma~\ref{claim:uniformity}, respectively.

\section{Preliminaries and Notation}\label{sec:prelim}

For the rest of the paper, $r\in\NN_{\geq 3}$ and $d>0$ will be fixed.
In particular, we consider these to be constant, so if we say, for example,
that $x = O(n)$, we mean that there exists a constant $C=C(r,d)$
such that $x \le Cn$.
By the notation $x=a\pm b$ we mean that $a-b\leq x\leq a+b$.
Similarly, the notation $x = (a\pm b)c$ means that $(a-b)c \le x \le (a+b)c$.
We will omit floors and ceilings whenever these do not significantly affect the argument. \\

As mentioned in Section~\ref{def:basicdefinitions}, the solution set of the fixed-point equation~\eqref{eq:fixedpointequation}
changes its behaviour at $d=d^*$. More precisely we have the following. 

\begin{claim}\label{claim:behaviouroffixedpointsol}\mbox{ }
\begin{enumerate}[label=\textnormal{\textbf{(F\arabic*)}}]
\item If $d<d^*$, then $\intvarsurvlim=0$.
\item If $d>d^*$, then there is a unique positive solution to~\eqref{eq:fixedpointequation}.
\end{enumerate}
\end{claim}
We defer the (rather technical) proof of this claim to Appendix~\ref{app:analysisfixedpoint}.

Furthermore we will often use the following alternative relation between $\intvarsurvlim$ and $\intfactsurvlim$.
\begin{equation}\label{eq:relatingrhoswithexpfunction}
1-\intvarsurvlim\numeq{\eqref{eq:fixedpointequation}}F(1-\intvarsurvlim)=\exp\left(-d\left(1-(1-\intvarsurvlim)^ {r-1}\right)\right)\numeq{\eqref{eq:relatingtherhos}}\exp(-d\intfactsurvlim).
\end{equation}

\subsection{Large deviation bounds}

In this section, we collect some standard large deviation results which will be needed later. We will use the following Chernoff bound
(see e.g.~\cite[Theorem 2.1]{JansonLuczakRucinskiBook}).
\begin{lemma}[Chernoff]\label{lem:chernoffbounds}
	If $X\sim \Bi(N,p)$, then for any $s> 0$
	\begin{equation*} \label{eqn:chernoffstd}
	\mathbb{P}(|X-Np|\geq s)\leq 2\cdot\exp\left(-\frac{s^2}{2\left(Np+\frac{s}{3}\right)}\right).
	\end{equation*}
	
\end{lemma}
This bound is less precise than the form in~\cite{JansonLuczakRucinskiBook} since we have combined the upper and lower tail bounds for simplicity.
We will also need an Azuma-Hoeffding inequality---the following lemma is taken from~\cite{JansonLuczakRucinskiBook}, Theorem $2.25$.
\begin{lemma}[Azuma-Hoeffding]\label{lem:Azuma}
Let $X$ be a real-valued random variable.
If $(X_i)_{1\leq i\leq n}$ is a martingale with $X_n=X$ and $X_0=\expec(X)$,
and for every $1\leq i\leq n$ there exists a constant $c_i>0$ such that
\[|X_i-X_{i-1}|\leq c_i\]then, for any $s>0$,
\begin{align}
    \prob(|X-\expec(X)|\geq s) & \leq 2\cdot\exp\left(-\frac{s^2}{2\sum_{i=1}^nc_i^2}\right).\nonumber
\end{align}
\end{lemma}
Again, the bound in~\cite{JansonLuczakRucinskiBook} is a bit stronger than what we state here since for simplicity we have combined upper and lower tail bounds.

\section{Factor graphs}\label{sec:factorgraphs}

There is a natural representation of a hypergraph as a bipartite graph known as a \emph{factor graph},
which is a well-known concept in literature (see e.g.~\cite{MezardMontanariBook}).

\begin{definition}[Factor graph]
Given a hypergraph $H$, the \emph{factor graph} $G=G(H)$ of $H$ is a bipartite graph on vertex classes $\cV\coloneqq V(H)$ and $\cF\coloneqq E(H)$,
where $v \in \cV$ and $a \in \cF$ are joined by an edge in $G$ if and only if $v \in a$. In other words, the vertices of $G$ are the vertices and edges of $H$,
and the edges of $G$ represent incidences.

To avoid confusion, we refer to the vertices of a factor graph as \emph{nodes}. In particular, the nodes in $\cV$ are called \emph{variable nodes}
and the nodes in $\cF$ are called \emph{factor nodes}.\footnote{In some contexts in the literature,
factor nodes may be called \emph{functional nodes} or \emph{constraint nodes}.} We define
$$
\randfact\coloneqq G(\cH),
$$
i.e.\ the factor graph of the $r$-uniform binomial random hypergraph $H^r(n,p)$.
\end{definition}

Note that if $H$ is an $r$-uniform hypergraph, then the factor nodes of $G(H)$ all have degree~$r$.
We will need the following basic fact about the number of factor nodes in $\randfact$. The proof is a simple application of a Chernoff bound,
and appears in Appendix~\ref{app:numberfactornodes} for completeness.

\begin{proposition}\label{prop:numberfactornodes}
Let $d>0$ be a constant and let $p=\frac{(1+o(1))d}{\binom{n}{r-1}}$. Then there exists a function $\omega_0=\omega_0(n)$ with
$\omega_0\xrightarrow{n\to\infty}\infty$ such that \whp\ the number $m$ of factor nodes
in $\randfact$ satisfies
\[
m=\left(1\pm \frac{1}{\omega_0}\right)\frac{dn}{r}.
\]
\end{proposition}

It will be more convenient to study the factor graph than the original hypergraph---in order to do this,
we need to understand what the structure corresponding to the loose core looks like in the factor graph.
We first define the loose core
of a factor graph and subsequently observe that it does indeed correspond to the loose core of the hypergraph (Definition~\ref{def:loosecore}).

\begin{definition}[Loose core]\label{def:loosecorefactorgraph}
The \emph{loose core} $C=\coreg$ of a factor graph $G$ is the maximal subgraph of $G$ such that each factor node of $C$
has degree $r$ in $C$ and furthermore:
\begin{enumerate}[label=\textnormal{\textbf{(C\arabic*')}}]
\item\label{item:isolatednodesfactorgraph} $C$ contains no isolated variable nodes;
\item\label{item:degreeconditionfactorgraphs} Each factor node in $C$ is adjacent to at least two variable nodes of degree at least two in $C$.
\end{enumerate}
\end{definition}

\begin{proposition}\label{prop:loosecorecorrespondence}
Given an $r$-uniform hypergraph $H$, the loose core $C_G$ of the factor graph $G=G(H)$ of $H$ is
identical to the factor graph of the loose core $C_H$ of $H$.
\end{proposition}

\begin{proof}
The condition that each factor node of $C=\coreg$ has degree $r$ in $C$
means that $C$ corresponds to a subhypergraph of $H$ (i.e.\ no edge of
$H$ has a vertex removed from it without itself being removed).
Since variable nodes of $G$ correspond to vertices of $H$, condition \ref{item:isolatednodesfactorgraph} in Definition~\ref{def:loosecorefactorgraph}
corresponds precisely to \ref{item:loosecorefirstcond} in Definition~\ref{def:loosecore}.
Furthermore, condition~\ref{item:degreeconditionfactorgraphs} in Definition~\ref{def:loosecorefactorgraph} is directly analogous
to condition~\ref{item:loosecoresecondcond} in Definition~\ref{def:loosecore}.
\end{proof}

In view of Proposition~\ref{prop:loosecorecorrespondence},
rather than studying the loose core of the hypergraph, we can study the loose core of the corresponding factor graph instead.
In fact, even more convenient than this is a slightly different structure.

\begin{definition}[Reduced core]\label{def:reducedcore}
The \emph{reduced core} $R=\corer$ of a factor graph $G$ is the maximal subgraph of $G$
with no nodes of degree~$1$.
\end{definition}

Note that the reduced core is very similar to the $2$-core of $G$---the only difference
is that we do not delete isolated nodes, so all original nodes are still present. This will be convenient
since it means that all nodes have a well-defined degree within the reduced core
(and have degree zero if and only if they are not in the $2$-core of $G$).
Similarly we will want to describe degree distributions within the loose core,
but also incorporating nodes which are in fact not contained in the loose core.
To avoid confusion and abuse of terminology, we define the \emph{padded core}.

\begin{definition}[\Paddedcore]\label{def:paddedcore}
The \emph{\paddedcore} $P=\pc$ of a factor graph $G$ is the subgraph of $G$ whose
nodes are the nodes of $G$ and whose edges are the edges of $\coreg$.
\end{definition}
In other words, the \paddedcore\ $\pc$ is identical to the loose core $\coreg$ except that
all nodes of $G$ are still present. Equivalently, $\pc$ is the maximal subgraph of $G$
in which each non-isolated factor node has degree~$r$ and is adjacent to at least two
variable nodes of degree at least~$2$. The following observation motivates both our definition
of the \paddedcore\ and the interpretation of $\mu_0$ as the proportion of vertices
of $H^r(n,p)$ which do not lie in the loose core.

\begin{remark}\label{rem:paddedcoredegs}
For each $j\in \NN$,
the proportion of variable nodes of $G=G^r(n,p)$ which have degree $j$ in the \paddedcore\ $\pc$ is $\mu_j$.
\end{remark}

It is important to observe that, if we have found the reduced core, it is very easy to reconstruct the \paddedcore,
and hence also the loose core. 
Let $\cF_R$ be the set of non-isolated factor nodes of the reduced core $\corer$ and let $\pcprime$
be the factor graph whose nodes are the nodes of $G$ and whose edges are all edges of $G$ incident to $\cF_R$.
In other words, $\pcprime$ is the factor graph obtained from $\corer$ by adding back in all
edges of $G$ attached to non-isolated factor nodes of $\corer$.

\begin{proposition}\label{prop:reducedtoloose}
$\pcprime = \pc$.
\end{proposition}

\begin{proof}
Let $R_1$ denote the graph obtained from the padded core $\pc$ of $G$ by removing
all edges incident to leaves (which must be variable nodes).
Note that, since any non-isolated factor node in $\pc$ has at least two neighbours of
degree at least two, the same is still true in $R_1$.
For the sake of intuitive notation, we also denote
$$
R_2 \coloneqq \corer, \qquad
P_1 \coloneqq \pc, \qquad P_2 \coloneqq \pcprime.
$$
Our goal is to show that $P_1=P_2$.

Let us observe that  $P_1$ can be obtained from $R_1$ by the same operation
with which $P_2$ is obtained from $R_2$, namely by adding in
edges of $G$ incident to non-isolated factor nodes.

We next observe that $R_1$ is a subgraph of $G$ with no nodes of degree $1$,
and therefore $R_1 \subseteq R_2=\corer$, by the maximality of $\corer$.
Since the operation constructing $P_i$ from $R_i$ is inclusion-preserving, $R_1 \subseteq R_2$ implies that $P_1 \subseteq P_2$. 

It therefore remains to prove that $P_2 \subseteq P_1$. 
To this end, we observe that certainly $P_2=\pcprime$ is a subgraph of $G$ in which each non-isolated factor node is
in the $2$-core of $G$, and therefore
adjacent to at least two variable nodes of degree at least two. Furthermore each non-isolated
factor node of $P_2$ has degree $r$ in $C_2$, and since
$P_1=\pc$ is the maximal subgraph with these two properties, we have $P_2 \subseteq P_1$, as required.
\end{proof}

Let us observe one further fact about the transformation from the reduced core $\corer$ to the padded core $\pc=\pcprime$:
although this seemed to be dependent on the initial factor graph $G$, in fact the operation
simply involves connecting non-isolated factor nodes of $\corer$ to (distinct) isolated variable nodes until each factor node has degree precisely~$r$.
This means that given $\corer$, by Proposition~\ref{prop:reducedtoloose} we can describe $\pc$ (and therefore also the loose core $\coreg$)
entirely, up to the assignment of which
nodes are leaves. In other words, $\corer$ already contains all of the ``essential'' information
of both $\pc$ and $\coreg$. It will therefore be enough to study $\corer$ rather than $\pc$ or $\coreg$, and this turns out to be simpler.

Now the main results of this paper are implied by the following theorem about the reduced core $\corer$
of the factor graph $G= G^r(n,p)$ of the $r$-uniform binomial random hypergraph.

For a non-negative real number $\lambda$,
let us denote by $\widetilde \Po(\lambda)$ the distribution of a random variable $X$ satisfying
$$
\Pr(X=j) = \begin{cases}
\Pr(\Po(\lambda)\le 1) & \mbox{if }j=0, \\
0 & \mbox{if }j=1,\\
\Pr(\Po(\lambda)=j) & \mbox{if } j \ge 2.
\end{cases}
$$
In other words, the $\widetilde\Po$ distribution is identical to the $\Po$ distribution except
that values of $1$ are replaced by $0$. We define the $\widetilde\Bi$ distribution analogously.

\begin{theorem}\label{thm:factor:reducedcoredegs}
Let $r,d,p,\intvarsurvlim,\intfactsurvlim$ be as in Section~\ref{def:basicdefinitions} and let $G= G^r(n,p)$, i.e.\ the factor graph of $\cH$.
For each $j\in\NN$, let $\varprop{j}$ and $\factprop{j}$ be the proportion
of variable nodes and factor nodes of $G$ respectively which have degree $j$ in the reduced core $\corer$ of $G$.
Then there exists a function $\error = \error(n) =o(1)$
such that
\whp\ for any constant $j \in \NN$ we have
$$
\varprop{j} = \Pr(\widetilde\Po(d\intfactsurvlim)=j)\pm\error
$$
and
$$
\factprop{j} = \Pr(\widetilde\Bi(r,\intvarsurvlim)=j)\pm\error.
$$
\end{theorem}
In other words, in terms of their degrees in $\corer$, variable nodes and factor nodes
have degree distributions which are asymptotically those of a $\widetilde \Po(d \intfactsurvlim)$
and a $\widetilde \Bi(r,\intvarsurvlim)$ distribution respectively.

The proof of this theorem will form the main body of the paper. In Section~\ref{sec:peeling} we will prove how Theorem~\ref{thm:factor:reducedcoredegs} follows from two auxiliary statements,
stating that for some large integer $\ell$
the proportions of variable and factor nodes of degree $j$ in the graph obtained after $\ell$ rounds of a peeling process
are approximately the values given in Theorem~\ref{thm:factor:reducedcoredegs}
(Lemma~\ref{lem:mainlemma1}),
and furthermore not many nodes are deleted after round $\ell$ (Lemma~\ref{lem:mainlemma2}). 

\section{Back to hypergraphs: Proofs of main results}\label{sec:proofofmainresults}
\noindent We now show how all of the results of Section~\ref{sec:mainresults} follow from Theorem~\ref{thm:factor:reducedcoredegs}. First we deduce our result on the asymptotic degree distribution of vertices in the loose core of $\cH$.
\begin{proof}[Proof of Theorem~\ref{thm:mainresultdegrees}]
We will apply Theorem~\ref{thm:factor:reducedcoredegs} to provide us with a function $\eps$,
and we will prove Theorem~\ref{thm:mainresultdegrees} with $\eps'\coloneqq\sqrt{\eps} + \frac{1}{\sqrt{\omega_0}}$, where
$\omega_0=\omega_0(n)$ is the function given by 
Proposition~\ref{prop:numberfactornodes}.

For convenience, for any $j\in\NN$, let us define
$$
\mu_j' \coloneqq \begin{cases}
\Pr(\Po(d\intfactsurvlim)=j) & \mbox{if } j \ge 2;\\
\oldbeta \cdot \Pr(\Po(d\intfactsurvlim)=j) & \mbox{if } j=1;\\
\Pr(\Po(d\intfactsurvlim)=0) + (1-\oldbeta)\cdot \Pr(\Po(d\intfactsurvlim)=1) & \mbox{if } j=0.
\end{cases}
$$
In other words, $\mu_j'$ is the ``idealised version'' of $\mu_j$,
and our goal is simply to prove that whp, for each $j\in\NN$ we have $\mu_j = \mu_j' \pm \eps$.
Similarly we also define
$$
\varpropp{j} \coloneqq \Pr(\widetilde{\Po}(d\intfactsurvlim)=j),
$$
so by Theorem~\ref{thm:factor:reducedcoredegs} we have $\varprop{j} = \varpropp{j} \pm \eps$
whp for each $j\in\NN$. The proof of Theorem~\ref{thm:mainresultdegrees}
now simply consists of relating the $\mu_j$ to the $\varprop{j}$,
relating the $\mu_j'$ to the $\varpropp{j}$ and applying Theorem~\ref{thm:factor:reducedcoredegs}.
Note that it follows instantly from the definitions
that $\mu_j' = \varpropp{j}$ for $j\in\NN_{\geq 2}$. We will split the proof into three cases.

\vspace{0.2cm}
\noindent \textbf{Case 1: $j\ge 2$.}\\
We start by showing that $\mu_j=\varprop{j}$. Observe that by Remark~\ref{rem:paddedcoredegs},
$\mu_j$ is simply
the proportion of variable nodes with degree $j$ in the \paddedcore\ $\pc$ of $G=\randfact$.
Theorem~\ref{thm:factor:reducedcoredegs} tells us the degrees of variable and factor
nodes in the reduced core $\corer$ of $G$. By Proposition~\ref{prop:reducedtoloose},
moving from $\corer$ to $\pc$ means that 
we connect all non-isolated factor nodes of $\corer$ to their original neighbours in $G$,
and any variable nodes which receive additional incident edges in this process have their degrees changed from~$0$ to~$1$.
It follows that for $j\ge 2$, the proportion $\mu_j$
of variable nodes in $G$ with degree $j$ in the \paddedcore\ $P_G$ is precisely equal to $\varprop{j}$,
the proportion of variable nodes in $G$ with degree $j$ in the reduced core $\corer$.
Therefore 
\[
\mu_j=\varprop{j}\numeq{\text{Th.}~\ref{thm:factor:reducedcoredegs}}\varpropp{j} \pm \eps=\mu_j'\pm \eps,
\]
and the statement of Theorem~\ref{thm:mainresultdegrees} is certainly true for $j\ge 2$ (indeed, we have proved something stronger since $\eps<\eps'$).

\vspace{0.2cm}
\noindent \textbf{Case 2: $j=1$.}\\
To prove the case $j=1$, we need to check how many isolated variable nodes become leaves when moving from $\corer$ to $\pc$.
Since by Proposition~\ref{prop:reducedtoloose} every factor node of $\corer$ with degree $j\ge 2$ has $r-j$ leaves connected to it, and since whp
the number $m$ of factor nodes in total is $m=\left(1\pm\frac{1}{\omega_0}\right)\frac{dn}{r}$ for some growing function
$\omega_0\xrightarrow{n\to\infty}\infty$ by Proposition~\ref{prop:numberfactornodes}, whp
the number of leaves added, which is simply $\mu_1 n$, satisfies
\begin{align}\label{eq:leaves1}
\mu_1 n=\sum_{j=2}^r (r-j) \factprop{j}  m
& = \left(1\pm\frac{1}{\omega_0}\right)\frac{dn}{r} \sum_{j=2}^r (r-j)\left(\Pr(\widetilde\Bi(r,\intvarsurvlim)=j)\pm \eps\right)\nonumber \\
& = \frac{dn}{r} \sum_{j=2}^r (r-j)\Pr(\widetilde\Bi(r,\intvarsurvlim)=j) \pm \frac{\eps'n}{2},
\end{align} 
where the last line follows since $\frac{1}{\omega_0},\eps = o\left(\sqrt{\eps}+\frac{1}{\sqrt{\omega_0}}\right)=o(\eps')$.
The sum can be estimated using the definition of the $\widetilde\Bi$ distribution
and equations~\eqref{eq:relatingtherhos} and~\eqref{eq:relatingrhoswithexpfunction}:
\begin{align*}
&\sum\nolimits_{j=2}^r (r-j)\Pr(\widetilde\Bi(r,\intvarsurvlim)=j) \\
& \hspace{2cm} \numeq{\phantom{\eqref{eq:relatingtherhos},\eqref{eq:relatingrhoswithexpfunction}}} \sum\nolimits_{j=0}^r (r-j)\Pr(\Bi(r,\intvarsurvlim)=j)  - r(1-\intvarsurvlim)^r - (r-1)r\intvarsurvlim(1-\intvarsurvlim)^{r-1}\\
& \hspace{2cm} \numeq{\phantom{\eqref{eq:relatingtherhos},\eqref{eq:relatingrhoswithexpfunction}}} r (1-\intvarsurvlim)\left( 1 - (1-\intvarsurvlim)^{r-1} - (r-1)\intvarsurvlim(1-\intvarsurvlim)^{r-2} \right)\\
& \hspace{2cm} \numeq{\eqref{eq:relatingtherhos},\eqref{eq:relatingrhoswithexpfunction}} r \exp(-d\intfactsurvlim)\left(\intfactsurvlim - (r-1)\intvarsurvlim(1-\intvarsurvlim)^{r-2} \right).
\end{align*}
Substituting this into~\eqref{eq:leaves1} gives
\begin{align}\label{eq:leaves2}
\mu_1
&  = d\exp(-d\intfactsurvlim) \left(\intfactsurvlim - (r-1)\intvarsurvlim(1-\intvarsurvlim)^{r-2}\right) \pm \eps'/2.
\end{align}
On the other hand, we have
\begin{align*}
\mu_1' = \oldbeta \cdot\Pr(\Po(d\intfactsurvlim)=1) & = \left(1- \frac{(r-1)\intvarsurvlim (1-\intvarsurvlim)^{r-2}}{\intfactsurvlim}\right) d\intfactsurvlim \exp(-d\intfactsurvlim)\\
& = d\exp(-d\intfactsurvlim)\left(\intfactsurvlim - (r-1)\intvarsurvlim (1-\intvarsurvlim)^{r-2}\right),
\end{align*}
which combined with~\eqref{eq:leaves2} tells us that
\begin{equation}\label{eq:deg1proportion}
\mu_1 = \mu_1' \pm \eps'/2,
\end{equation}
which is in fact slightly stronger than required.

\vspace{0.2cm}
\noindent \textbf{Case 3: $j=0$.}\\
Finally to prove the statement for $j=0$, note that
$\mu_0 = \varprop{0} - \mu_1$ (deterministically).
Furthermore, we have $\sum_{j=0}^\infty \mu_j' = \sum_{j=0}^\infty \varpropp{j} =1$,
and we have already observed that $\mu_j'=\varpropp{j}$ if $j\ge 2$,
and therefore $\mu_0' + \mu_1' = \varpropp{0} + \varpropp{1}$.
Observing also that $\varpropp{1}=0$, we deduce that
$\mu_0' = \varpropp{0} - \mu_1'$.
Therefore, applying Theorem~\ref{thm:factor:reducedcoredegs} (for $j=0$) and~\eqref{eq:deg1proportion},
we obtain
$$
\mu_0 = \varprop{0}-\mu_1 = \varpropp{0} \pm \eps - (\mu_1'\pm \eps'/2) = \mu_0'\pm \eps'
$$
as required.
\end{proof}

With a little more calculation we can also determine
the number of vertices and edges in the loose core, and therefore also prove Theorem~\ref{thm:mainresultorder}.

\begin{proof}[Proof of Theorem~\ref{thm:mainresultorder}]
Observe that the number of vertices in the loose core of $H=H^r(n,p)$ is simply the number of variable nodes
of $G=G(H)$ which have degree at least one in the \paddedcore\ of $G$,
and thus the proportion of such vertices is $1-\mu_0$ (see Remark~\ref{rem:paddedcoredegs}). By Theorem~\ref{thm:mainresultdegrees}, \whp
\begin{align*}
1-\mu_0 & \numeq{\phantom{\eqref{eq:relatingtherhos},\eqref{eq:eta},\eqref{eq:relatingrhoswithexpfunction}}} 1-\exp(-d\intfactsurvlim) - d\intfactsurvlim\exp(-d\intfactsurvlim)(1-\oldbeta)+o(1) \\
& \numeq{\eqref{eq:relatingtherhos},\eqref{eq:eta},\eqref{eq:relatingrhoswithexpfunction}}\intvarsurvlim - d(1-\intvarsurvlim)(r-1) \intvarsurvlim(1-\intvarsurvlim)^{r-2}+o(1)\\
& \numeq{\phantom{\eqref{eq:relatingtherhos},\eqref{eq:eta},\eqref{eq:relatingrhoswithexpfunction}}}\intvarsurvlim (1 - d(r-1)(1-\intvarsurvlim)^{r-1})+o(1)=\propnumvs+o(1),
\end{align*}
precisely as stated in Theorem~\ref{thm:mainresultorder}.

The number of edges in the loose core of $H$ is the number of factor nodes with degree at least $1$ in $\corer$,
which is $(1-\factprop{0})m$, where recall that $m$ denotes the total number of factor nodes of $G$.
Applying Theorem~\ref{thm:factor:reducedcoredegs} to estimate $\factprop{0}$
and Proposition~\ref{prop:numberfactornodes} to estimate $m$,
we deduce that
\whp\ the number of edges in the loose core is
\begin{align*}
\left(1-\factprop{0}\right)m
&=\left(1-(1-\intvarsurvlim)^r-r\intvarsurvlim(1-\intvarsurvlim)^{r-1}\pm o(1)\right)\frac{(1+o(1))dn}{r}
\numeq{\eqref{eq:definitionofeta}}\left(\oldeta+o(1)\right)n,
\end{align*} 
as claimed.
\end{proof}

Now we can also prove the bound on the length of the longest loose cycle in Theorem~\ref{thm:mainresultcycle}.

\begin{proof}[Proof of Theorem~\ref{thm:mainresultcycle}]
Let us observe that for any loose cycle in $H=H^r(n,p)$, the edges and the vertices which lie in two edges form
a cycle in the factor graph $G=G(H)$, which must clearly lie within the reduced core $\corer$ of $G$. Thus the
length of the loose cycle is bounded both by the number of variable nodes and the number of factor nodes
which are not isolated
in $\corer$. In other words, the length $L_C$ of the longest loose cycle (deterministically) satisfies
\begin{equation}\label{eq:cyclelengthmin}
L_C\le \min\left\{(1-\varprop{0})n \; , \; (1-\factprop{0}) m\right\}.
\end{equation}
By Proposition~\ref{prop:numberfactornodes} we have that \whp\ $m=\left(1+o(1)\right)\frac{dn}{r}$.
Observe also that by Theorem~\ref{thm:factor:reducedcoredegs}, \whp\ $\varprop{0}$ is asymptotically
\begin{align*}
\Pr(\widetilde{\Po}(d\intfactsurvlim)=0) = \prob(\Po(d\intfactsurvlim)\leq 1) =  \exp(-d\intfactsurvlim)(1+d\intfactsurvlim)=1-\gamma,
\end{align*}
while whp $\factprop{0}$ is asymptotically
\begin{align*}
\Pr(\widetilde{\Bi}(r,\intvarsurvlim)=0) = \prob(\Bi(r,\intvarsurvlim)\leq 1) = (1-\intvarsurvlim)^r + r\intvarsurvlim(1-\intvarsurvlim)^{r-1}=1-\frac{\oldeta r}{d}.
\end{align*}
Substituting these values into~\eqref{eq:cyclelengthmin} gives the bound in Theorem~\ref{thm:mainresultcycle}.
\end{proof}

Our next goal is to prove the remaining results of Section~\ref{sec:mainresults},
for which we will need to relate $L_P$ and $L_C$. To do this, we use a standard sprinkling argument.

\begin{lemma}\label{lem:standardsprinkling}
Let $\omega = \omega(n)$ be any function and
$p_1=p_1(n)$ and $p_2=p_2(n)$ be any probabilities satisfying
\begin{enumerate}
\item $p_1 \le \left(1+\frac{1}{\omega}\right)p_1 \le p_2;$
\item $p_1 n^r /\omega \to \infty$.
\end{enumerate}
Suppose that $H_1 \sim H^r(n,p_1)$ and $H_2 \sim H^r(n,p_2)$ are coupled in such a way that $H_1 \subset H_2$.
For $i=1,2$, let $L_P^{(i)},L_C^{(i)}$ denote the length of the longest loose path and loose cycle, respectively, in $H_i$. Then \whp\ \[L_C^{(2)} \ge L_P^{(1)}+o(n).\]
\end{lemma}
We defer the proof of this lemma to Appendix~\ref{app:sprinkling}.
The following slightly different form will be a little more convenient to apply.
We omit the proof, which is elementary given Lemma~\ref{lem:standardsprinkling}. \pagebreak
\begin{corollary}\label{cor:sprinkling}
Given the setup of Lemma~\ref{lem:standardsprinkling}, the following hold.
\begin{enumerate}
\item If there exists a constant $\zeta_1$ such that \whp\ $L_P^{(1)} \ge (\zeta_1 +o(1))n$, then \whp\ $L_C^{(2)} \ge (\zeta_1 +o(1))n$.
\item If there exists a constant $\zeta_2$ such that \whp\ $L_C^{(2)} \le (\zeta_2 +o(1))n$, then \whp\ $L_P^{(1)} \le (\zeta_2 +o(1))n$.
\end{enumerate}
\end{corollary}

We also need a further technical result which states that the parameters $\oldeta,\gamma$,
with which we bound $L_C$ in Theorem~\ref{thm:mainresultcycle}, are continuous in $p$
(except at the threshold $p=d^*/\binom{n-1}{r-1}$).
Let $r,d,p$ be as in Section~\ref{def:basicdefinitions}. 
Let $p'=(1+1/\omega)p$ for a function $\omega=\omega(n)\to\infty$ but $\omega=o(\log n)$.
The following lemma shows that if we replace $p$ by $p'$,
the parameters $\oldeta,\gamma$ remain essentially the same. The (technical) proof can be found in Appendix~\ref{app:analysisfixedpoint}.
\begin{lemma}\label{lem:continuity}
Let $\oldeta,\gamma$ be defined as in~\eqref{eq:definitionofeta} and~\eqref{eq:definitionofgamma},
and let $\oldeta',\gamma'$ be defined similarly but with $p'$ in place of $p$. If $d\neq d^*,$ then 
\[\min\{\oldeta',\gamma'\}=\min\{\oldeta,\gamma\}+o(1).\]
\end{lemma}

We can now bound the length of the longest loose path in $H^r(n,p)$.
\begin{proof}[Proof of Corollary~\ref{cor:upperboundlongestpath}]
Let us set $\omega = 1/(\log n)$, set $p_1=p$ and set $p_2 = \left(1+\frac{1}{\omega}\right)p_1$.
It is easy to check that these parameters satisfy the assumptions of Lemma~\ref{lem:standardsprinkling},
and therefore also of Corollary~\ref{cor:sprinkling}.
Theorem~\ref{thm:mainresultcycle} applied to $H_2 \sim H^r(n,p_2)$ implies that \whp\ $L_C^{(2)} \le (\min\{\oldeta_2,\gamma_2\} + o(1))n$,
where $\oldeta_2,\gamma_2$ are defined analogously to $\oldeta,\gamma$, but with $p_2 = (1+1/\omega)p$ in place of $p$.
Furthermore, Lemma~\ref{lem:continuity} implies that $\min\{\oldeta_2,\gamma_2\} = \min\{\oldeta,\gamma\}+o(1)$,
so we deduce that \whp\ $L_C^{(2)} \le (\min\{\oldeta,\gamma\} + o(1))n$.
Finally, Corollary~\ref{cor:sprinkling} then implies that \whp\ $L_P=L_P^{(1)} \le (\min\{\oldeta,\gamma\} + o(1))n$,
as required.
\end{proof}

By applying Corollary~\ref{cor:upperboundlongestpath} shortly beyond the phase transition
threshold, we are able to prove Corollary~\ref{cor:shortlyafterphasetransition}.
\begin{proof}[Proof of Corollary~\ref{cor:shortlyafterphasetransition}]
Since $\min\{\oldeta,\gamma\}\leq \gamma$ it suffices to show that 
\[\gamma\leq \frac{2\eps^2}{(r-1)^2}+O(\eps^3).\] (In fact a similar computation for $\oldeta$ gives exactly the same result.)
By definition
\begin{align}
    \gamma &\numeq{\eqref{eq:definitionofgamma}} 1-\exp(-d\intfactsurvlim)-d\intfactsurvlim \exp(-d\intfactsurvlim) \nonumber \\
    & = 1-\left(1-d\intfactsurvlim+\frac{d^2\intfactsurvlim^2}{2}+O\left(\intfactsurvlim^3\right)\right)-d\intfactsurvlim\left(1-d\intfactsurvlim+O\left(\intfactsurvlim^2\right)\right) \nonumber \\
    &= \frac{d^2\intfactsurvlim^2}{2}+O\left(\intfactsurvlim^3\right) \label{eq:gammadefapprox}
\end{align}
Recall from~\eqref{eq:relatingtherhos} that $\intfactsurvlim$ was defined as a function of $\intvarsurvlim$,
which itself was defined as the largest solution of the fixed-point equation~\eqref{eq:fixedpointequation}.
We therefore need to estimate $\intvarsurvlim$.
From~\eqref{eq:fixedpointequation} we obtain
\begin{align*}
    \rho &= \frac{-d(r-1)+1}{(-\frac{1}{2}-\frac{d}{2}(r-1)(r-2))}+O\left(\rho^2\right).\label{eq:rhoapproximation}
\end{align*}

\noindent Substituting $d=\frac{1+\eps}{r-1}$ gives
\begin{equation*}
    \rho=\frac{2\eps}{1+(1+\eps)(r-2)}+O\left(\rho^2\right)=\frac{2\eps}{r-1+O(\eps)}+O\left(\rho^2\right)=\frac{2\eps}{r-1}+O\left(\rho^2\right).
\end{equation*}
In particular this implies that there exists a solution $\rho=\frac{2\eps}{r-1}+O(\eps^2)$
of the fixed point equation~\eqref{eq:fixedpointequation},
and by Claim~\ref{claim:behaviouroffixedpointsol} this is the unique positive solution
and therefore
$\intvarsurvlim=\frac{2\eps}{r-1}+O(\eps^2)$.
Substituting this into~\eqref{eq:relatingtherhos} we obtain
\begin{equation*}
\intfactsurvlim=1-\left(1-\frac{2\eps}{r-1}+O(\eps^2)\right)^{r-1}=2\eps+O(\eps^2).
\end{equation*}
Substituting this into~\eqref{eq:gammadefapprox}, we obtain
\begin{align*}
    \gamma &= \frac{2\eps^2}{(r-1)^2}+O\left(\eps^3\right).\qedhere
\end{align*}
\end{proof}

In order to prove Theorem~\ref{thm:bestknownresultcycles},
we also need a lower bound on $L_C$.
We will use a result of~\cite{cooley2020longest},
which provides a lower bound on $L_P$ together with Lemma~\ref{lem:standardsprinkling} to relate $L_P$ and $L_C$.
More precisely, one special case (the supercritical regime for $j=1$) of~\cite[Theorem~4]{cooley2020longest} can be
reformulated (in a slightly weakened but much simplified way) as follows.

\begin{theorem}[\!\!\cite{cooley2020longest}]\label{thm:pathsresult}
Let  $L_P$ denote the length of the longest loose  path in $\cH$. 
For all $r\in \NN_{\ge 3}$ there exists $\eps_0 \in (0,1]$ such that for any function
$\eps=\eps(n)<\eps_0$ which satisfies
$\eps^5 n \xrightarrow{n\to \infty} \infty$, setting $\delta=\eps/\sqrt{\eps_0}$
the following holds. 
If $p=\frac{1+\eps}{(r-1)\binom{n-1}{r-1}}$,
then \whp
 \[
 (1 - \delta)\frac{\eps^2 n}{4(r-1)^2} \leq L_P \leq (1 + \delta)\frac{2 \eps n}{(r-1)^2}.
 \]
\end{theorem}

Note that Theorem~\ref{thm:pathsresult} allows for a wider range of $\eps$ than we consider in this paper,
in particular allowing $\eps$ to tend to zero sufficiently slowly. However, there is a $\Theta(\eps)$ gap
between the upper and lower bounds. Theorem~\ref{thm:bestknownresultcycles} improves the upper bound
and thus narrows the gap to just a constant factor.

\begin{proof}[Proof of Theorem~\ref{thm:bestknownresultcycles}]
The second and third inequalities are simply the statement of Corollary~\ref{cor:upperboundlongestpath},
so it remains to show that \whp
\[
\left(\frac{\eps^2}{4(r-1)^2}+O(\eps^3)\right)\cdot n\leq L_C.
\]
Note that we may assume that $\eps<\eps_0$, where $\eps_0$ is the parameter from
Theorem~\ref{thm:pathsresult},
since otherwise the $O(\eps^3)$ error term may in fact be the dominant term, and the result is trivial.

Let us set $p_2 = p$ and $p_1 = \left(1-\frac{1}{\log n}\right)p$. It is easy to check that these parameters
satsify the assumptions of Lemma~\ref{lem:standardsprinkling}, and therefore also of Corollary~\ref{cor:sprinkling}.

It is also clear that $p_1 = \frac{1+\eps_1}{(r-1)\binom{n-1}{r-1}}$, where $\eps_1 = \eps-\frac{1}{\log n}-\frac{\eps}{\log n}=\eps + O(\eps^2)$,
and
therefore the lower bound in Theorem~\ref{thm:pathsresult} (together with the observation that $\eps_1/\sqrt{\eps_0} = O(\eps_1)$)
states that \whp\ 
$$
L_P^{(1)} \ge \left(\frac{\eps_1^2}{4(r-1)^2} + O(\eps_1^3)\right)\cdot n = \left(\frac{\eps^2}{4(r-1)^2} + O(\eps^3)\right)\cdot n,
$$
and an application of Corollary~\ref{cor:sprinkling} completes the proof.
\end{proof}

\section{Peeling process}\label{sec:peeling}
Recall that for a given hypergraph $H$ the reduced core of the factor graph $G=G(H)$
is defined as the maximum subgraph with no nodes of degree one, which is similar to the $2$-core of $G$
except that isolated nodes are not deleted.
There is a simple peeling process to obtain the $2$-core of $G$ which is a standard procedure and has been used and analysed extensively in the literature.
We will consider the obvious adaptation of this process which obtains the reduced core rather than the $2$-core.

\begin{definition}[Peeling Process]\label{def:peeling}
In every round we check whether the factor graph has any nodes of degree one
and delete edges incident to such nodes. More precisely, we recursively define a sequence of graphs $(G_i)_{i\in\NN}$
where $G_0$ is the input graph and for $i\in\NN_{\geq 1}$, $G_i$ is the graph obtained from $G_{i-1}$
by removing all edges incident to nodes of degree one. We say that we \emph{disable} a node
if we delete its incident edges.

\end{definition} 
\noindent Note that there exists an $i_0$ such that $G_{i_0}=G_{i_0+k}=\corerzero$ for any $k\in\NN$. \pagebreak

We recall the definition of $\varprop{j}$ and $\factprop{j}$ in Theorem~\ref{thm:factor:reducedcoredegs} and observe that \[\varprop{j}\coloneqq\lim\limits_{\ell\to\infty}\varpropl{j}{\ell} \text{ \ and \ } \factprop{j}\coloneqq\lim\limits_{\ell\to\infty}\factpropl{j}{\ell},\] 
where $\varpropl{j}{\ell}, \factpropl{j}{\ell}$ are the proportions of variable nodes and factor nodes respectively which have degree $j$ in $G_{\ell}$ for $\ell\in\NN$. These limits exist since both $\left(\varpropl{j}{\ell}\right)_{\ell}$ and $\left(\factpropl{j}{\ell}\right)_{\ell}$ remain constant after a finite number of steps.

We will prove Theorem~\ref{thm:factor:reducedcoredegs} with the help of two lemmas. The first describes the asymptotic distribution of $\varpropl{j}{\ell}$ and $\factpropl{j}{\ell}$ for large $\ell$.
\begin{lemma}\label{lem:mainlemma1}
Let $r,d,\intvarsurvlim,\intfactsurvlim$ be as in Section~\ref{def:basicdefinitions}.
There exist an integer $\ell= \ell(n) \in \NN$ and a real number $\eps_1=\eps_1(n)=o(1)$
such that \whp, for any constant $j\in\NN$ 
\[
\varpropl{j}{\ell}=\prob(\widetilde \Po(d\intfactsurvlim)=j)\pm\eps_1
\]
and
\[
\factpropl{j}{\ell}=\prob(\widetilde \Bi(r,\intvarsurvlim)=j)\pm\eps_1.
\]
\end{lemma}

The second lemma states that $\varpropl{j}{\ell}$ and $\factpropl{j}{\ell}$ approximate $\varprop{j}$ and $\factprop{j}$, respectively.
\begin{lemma}\label{lem:mainlemma2}
Let $r,d$ be as in Section~\ref{def:basicdefinitions}.
For each $j\in \NN$,
let $\varprop{j},\factprop{j}$ be as defined in Theorem~\ref{thm:factor:reducedcoredegs},
let $\ell,\eps_1$ be as in Lemma~\ref{lem:mainlemma1} and set $\eps_2\coloneqq\sqrt{\eps_1}.$ Then \whp\ the peeling process will disable at most $\eps_2 n$ nodes after round $\ell$. In particular \whp, for any constant $j\in\NN$
\[
\varprop{j}=\varpropl{j}{\ell}\pm\eps_2
\]
and
\[
\factprop{j}=\factpropl{j}{\ell}\pm \frac{2\eps_2r}{d} .
\]
\end{lemma}
Before proving these two lemmas, we show how together they imply our main result. 
\begin{proof}[Proof of Theorem~\ref{thm:factor:reducedcoredegs}]
Let $\ell,\eps_1,\eps_2$ be as in Lemmas~\ref{lem:mainlemma1} and~\ref{lem:mainlemma2}.
Applying these two lemmas, \whp\ we have
$$
\varprop{j} \stackrel{\mbox{\tiny L.\ref{lem:mainlemma2}}}{=} \varpropl{j}{\ell} \pm \eps_2
\stackrel{\mbox{\tiny L.\ref{lem:mainlemma1}}}{=} \Pr(\widetilde \Po(d\intfactsurvlim)=j) \pm (\eps_1 + \eps_2).
$$
Similarly, \whp\ we have
$$
\factprop{j} \stackrel{\mbox{\tiny L.\ref{lem:mainlemma2}}}{=} \factpropl{j}{\ell} \pm \frac{2\eps_2 r}{d}
\stackrel{\mbox{\tiny L.\ref{lem:mainlemma1}}}{=} \prob(\widetilde \Bi(r,\intvarsurvlim\, )=j)\pm \left(\eps_1 + \frac{2\eps_2 r}{d}\right).
$$
The statement of Theorem~\ref{thm:factor:reducedcoredegs} follows by setting $\eps = \eps_1+\eps_2\max\{1,2r/d\}$.
\end{proof}

\section{\coredetectortwo\ Algorithm: Proof of Lemma~\ref{lem:mainlemma1}}\label{sec:algorithm}
\subsection{Main algorithm}
In this section we will introduce the \coredetectortwo\ algorithm,
which is related to the peeling process. To do so, we need to define some notation---this
notation could apply to any graph, but since we will need it specifically for factor graphs, we introduce it in this
(slightly restrictive) setting for clarity.
\begin{definition}\label{def:algorithmdefs}
Let $G$ be a factor graph with variable node set $\cV$ and factor node set $\cF$.
We denote by $d_G(u,v)$ the distance between two nodes $u,v\in \cV \cup \cF$, i.e.\ the number of edges in a shortest path between them.
For each $\ell\in\NN$ and each $w\in \mathcal{V}\cup\mathcal{F}$, we define
\[
D_{\ell}(w)\coloneqq\{u\in\mathcal{V}\cup\mathcal{F}:d_{G}(w,u)=\ell\}
\]
and
\[
d_{\ell}(w)\coloneqq|D_{\ell}(w)|.
\]
\pagebreak 

\noindent Let 
\[
D_{\leq\ell}(w)=\bigcup_{i=0}^\ell D_i(w)
\]
and
\[N_{\leq\ell}(w)\coloneqq G[D_{\leq\ell}(w)],\]  i.e.\ the subgraph of $G$ induced on $D_{\leq\ell}(w)$.
\end{definition}

We consider a procedure called \coredetectortwo.
Given a factor graph $G$ on node set $\cV \cup \cF$
and a node $w \in \cV \cup \cF$,
we consider the factor graph as being rooted at $w$. In particular, neighbours of a node $v$ which are at distance $d_G(v,w)+1$ from $w$
are called \emph{children} of $v$.
Starting at distance $\ell \in \NN$ and moving up towards the root $w$, we recursively delete any node with no (remaining) children;
Algorithm~\ref{algorithm:coredetector 2} gives a formal description of this procedure.
We will denote by $\hD_{\ell-i}(w)$ the set of nodes in $D_{\ell-i}(w)$ which survive round $i$ and let $\hd_{i}(w)\coloneqq|\hD_{i}(w)|$. 
\begin{algorithm}
	\DontPrintSemicolon
	\KwIn{Integer $\ell\in\NN$, node $w\in\cV\cup\cF$, factor graph $N_{\leq\ell+1}(w)$}
	\KwOut{$\hd_1(w)$}
    $\hD_{\ell+1}(w)=D_{\ell+1}(w)$\;
	\For{$1\leq i\leq \ell$}{
    $\hD_{\ell-i+1}(w)\gets D_{\ell-i+1}(w)\setminus\ \Big\{v:N(v)\cap\hD_{\ell-i+2}(w)=\emptyset\Big\}$\;
	$\hd_{\ell-i+1}(w)\gets|\hD_{\ell-i+1}(w)|$}
		\caption{\coredetector }
	\label{algorithm:coredetector 2}
\end{algorithm}

It is rather difficult to analyse the peeling process directly and it turns out that \coredetectortwo\ is easier to analyse
while also being closely related.
\coredetectortwo\ is intended to model the effect of the peeling process on the degree of $w$ after $\ell$ steps
(although note that \coredetectortwo\ does delete nodes rather than merely disabling them).
Note, however, that it does not mirror the peeling process precisely; some nodes may be disabled in the peeling process much
earlier than they are deleted in \coredetectortwo,
and some nodes may be deleted in \coredetectortwo\ even
though they are actually in the reduced core, and are therefore never disabled in the peeling process.
Nevertheless, we obtain the following important relation.
Recall that $G_{\ell}$ is the graph obtained after the $\ell$-th round of the peeling process (see Definition~\ref{def:peeling})
and that $d_{G_\ell}(w)$ is the degree of the node $w$ in $G_\ell$.

\begin{lemma}\label{lem:algvspeeling}
Let $\ell\in\NN_{\geq 1}$ and $w\in\cV\cup\cF$. If there are no cycles in $N_{\leq\ell+1}(w)$, then
the output $\hd_1(w)$ of \coredetectortwo\ 
with input $\ell,w$ and $N_{\le \ell+1}(w)$ satisfies
\[d_{G_{\ell}}(w) \begin{cases}
=\hd_1(w) & \text{if  }\hd_1(w)\neq 1, \\
\leq \hd_1(w) & \text{if  }\hd_1(w)=1.
\end{cases}
\]
\end{lemma}
\begin{proof}
For an upper bound,
we will show that if a given node $v\in\cV\cup\cF$ is deleted in round $i$ of \coredetectortwo, 
it must have been disabled at some round $i'\leq i$ in the peeling process for the reduced core.
In particular, by setting $i=\ell$ we immediately obtain
$d_{G_{\ell}}(w)\leq\hd_1(w).$
We prove the statement by induction on $i$.

For $i=1$, if a node $v$ is deleted in round one of \coredetectortwo, then $v$ had no children,
and therefore it has only one neighbour in $G=G_0$ (its parent in $N_{\le \ell+1}(w)$).
Thus $v$ will be disabled in round one
of the peeling process.
Now suppose $v$ is deleted in round $i\ge 2$ of \coredetector, which must mean that all its children
(if it had any) are deleted in step $i-1$ of \coredetector.
By the induction hypothesis, all its children are disabled by step at most $i-1$ of the peeling process
and so have degree $0$ in $G_{i-1}$.
Therefore $v$ itself has degree at most one in $G_{i-1}$ (from its parent)
and so will be disabled in round $i$ of the peeling process if it has not been disabled already.

It remains to prove that 
$d_{G_{\ell}}(w)\geq\hd_1(w)$
if $\hd_1(w)\geq 2$. Let $j\coloneqq\hd_1(w)\ge 2$ be the number of children of the root $w$ which survive \coredetectortwo.
Each such child must have a descendant in $D_{\ell+1}(w)$, otherwise it would not survive \coredetectortwo.
Thus we have $j$ paths of length $\ell+1$ which all meet at $w$, but are otherwise disjoint (since $N_{\le \ell+1}(w)$ contains no cycles).
By induction on $i$, we deduce that after $i$ rounds of the peeling process,
there are $j$ paths of length $\ell+1-i$ which meet only in $w$,
and in particular after $\ell$ rounds of the peeling process, $d_{G_\ell}(w) \ge j$, as required.
\end{proof}

From now on for the rest of this section, we will always have $G=G^r(n,p)$.
Observe that if $d_{G_\ell}(w)\in \{0,1\}$, then $d_{R_G}(w)=0$. However, if $\ell$ is sufficiently large
we can even say that
in $G_\ell$ the degree will almost always be $0$.

\begin{proposition}\label{prop:1rare}
For any integer-valued function $\ell=\ell(n) \xrightarrow{n\to \infty} \infty$ and node $w$ we have
$\Pr\left(d_{G_\ell}(w)=1\right) = o(1)$.
\end{proposition}

\begin{proof}
Let us first assume that $w$ is a variable node.
For any integer $i \ge 1$, let $\cV_i$ and $\cF_i$ be the set of variable nodes and factor nodes respectively which are disabled in round~$i$
of the peeling process.
It is an elementary fact about the peeling process that for any integer $i \ge 2$
we have
$|\cV_i| \le |\cF_{i-1}|$ and $|\cF_i| \le |\cV_{i-1}|$ (deterministically),
from which it follows that $|\cV_i|\le |\cV_{i-2}|$ for $i \ge 3$.
Therefore we have
$$
|\cV_\ell| \le |\cV_{\ell-2}| \le \ldots \le |\cV_{\ell -2\lfloor \frac{\ell-1}{2}\rfloor}|.
$$
Furthermore, the $\cV_i$ are all disjoint, and so we have
$$|\cV_\ell| \le \frac{n}{1+\lfloor \frac{\ell-1}{2}\rfloor} = O(n\ell^{-1}) = o(n)$$
deterministically. Therefore for large enough $n$ we have
$|\cV_\ell| \in K \coloneqq \left[n/\sqrt{\ell}\right]_0$ and so by symmetry we have
$$
\Pr\left(w \in \cV_\ell\right)  = \sum_{k \in K} \left( \Pr\left(|\cV_\ell|=k\right) \cdot \frac{k}{n} \right) \le \frac{1}{\sqrt{\ell}} \cdot \sum_{k \in K} \Pr\left(|\cV_\ell|=k\right) = o(1),
$$
as required.
The proof when $w$ is a factor node is essentially identical.
\end{proof}

\subsection{Analysis of \coredetectortwo}

We proceed with the analysis of \coredetectortwo\ and we will choose the parity of $\ell$ such that $D_{\ell+1}(w)$ consists of variable nodes,
i.e.\ if $w\in \cV$, then we will choose $\ell$ odd, and if $w \in \cF$ we will choose $\ell$ to be even.
This convention is merely for technical convenience since it ensures that we
know which type of nodes are being considered in round $i$ of \coredetector\ and thus avoid
a case distinction.

Let $w\in\cV\cup\cF$ and $\ell\in\NN_{\geq 1}$ be given. We say the event $E_w(\ell)$ holds if \emph{neither} of the following two events occur.  
\begin{enumerate}[label=\textnormal{\textbf{(E\arabic*)}}]
    \item $|D_{\leq \ell+1}(w)|\geq(\log n)^2$;
    \item $D_{\leq \ell+1}(w)$ contains a node which lies on a cycle of length at most $2\ell$.
\end{enumerate}
We will later condition on the event $\eone$ holding, and therefore need to know that it is very likely.

\begin{lemma}\label{lem:eventE}
For any function $\ell=o(\log\log n)$ and node $w \in \cV \cup \cF$,
$$
\Pr\left(\eone\right) \ge 1-\exp\left(-\Theta\left(\sqrt{\log n}\right)\right).
$$
Furthermore, \whp\ all but $o(n)$ nodes $w\in\cV\cup\cF$ satisfy $\eone$.
\end{lemma}

The (standard) proof appears in Appendix~\ref{app:eventE}.

Now let us define
$$
\td_1(w):=
\begin{cases}
\hd_1(w) & \mbox{if }\hd_1(w) \neq 1\\
0 & \mbox{if }\hd_1(w) =1.
\end{cases}
$$
In other words, $\td_1(w)$ is identical to $\hd_1(w)$ except that values of $1$ are replaced by $0$
(similar to the $\widetilde\Po$ and $\widetilde\Bi$ distributions compared to the $\Po$ and $\Bi$ distributions).
We can combine Lemmas~\ref{lem:eventE} and~\ref{lem:algvspeeling} and Proposition~\ref{prop:1rare}
to obtain the following.

\begin{corollary}\label{cor:Gldegalmosteq}
For any integer-valued function $\ell=\ell(n) \xrightarrow{n\to \infty} \infty$ which also satisfies $\ell = o(\log \log n)$ and any node $w$ we have
$$
\Pr\left(d_{G_\ell}(w) \neq \td_1(w)\right) = o(1).
$$
\end{corollary}

\begin{proof}
By Lemma~\ref{lem:algvspeeling}, the only cases in which $d_{G_\ell}(w)$ and $\td_1(w)$ can differ are if $\eone$ does not hold or if $d_{G_\ell}(w)=1$.
Thus by applying Lemma~\ref{lem:eventE} and Proposition~\ref{prop:1rare}, we obtain
\begin{align*}
\Pr\left(d_{G_\ell}(w) \neq \td_1(w)\right) & \le \Pr\left(\bar{E}_w(\ell)\right) + \Pr\left(d_{G_\ell}(w) =1\right)\\
& \le \exp\left(-\Theta\left(\sqrt{\log n}\right)\right) + o(1) = o(1).\qedhere
\end{align*}
\end{proof}

We next describe the survival probabilities of internal (i.e.\ non-root) variable and factor nodes in each round of \coredetector.
Recall that for any $i\in[\ell]$
the set $\hD_{\ell+1-i}(w)$ consists of nodes within $D_{\ell+1-i}(w)$ which survive the $i$-th round of \coredetectortwo. We define the recursions
\begin{align*}
    \intvarsurv{0} &= 1,\\
    \intfactsurv{t} &= \prob(\Bi(r-1,\intvarsurv{t-1})\geq 1),\\
    \intvarsurv{t} &= \prob(\Po(d\intfactsurv{t})\geq 1).
\end{align*}
\begin{lemma}\label{lem:survivalprobs}
Let $w\in\cV\cup\cF$ and $\ell$ be odd if $w\in\cV$ and even if $w\in\cF$. Let $t\in\NN$ with $\ 0\leq t\leq \frac{\ell+1 }{2}$ be given.
Conditioned on the event $\eone$:
\begin{enumerate}
\item
For each $u\in D_{\ell+1-2t}(w)$ independently of each other we have
\[
\prob[u\in\hD_{\ell+1-2t}(w)]=\intvarsurv{t}+o(1);
\]
\item
For each $a\in D_{\ell-2t}(w)$ independently of each other we have
\[
\prob[a\in\hD_{\ell-2t}(w)]=\intfactsurv{t+1}+o(1).
\]
\end{enumerate}
In particular:
\begin{enumerate}[(i)]
\item If $w \in \cF$ and $t_1=\ell/2$, then for each $u \in D_1(w)$ independently of each other,
$$
    \prob[u\in\hD_{1}(w)]=\intvarsurv{t_1}+o(1);
$$

\item\label{eq:probfactnodesurvives} If $w \in \cV$ and $t_2=(\ell+1)/2$, then for each $a\in D_1(w)$ independently of each other,
$$
    \prob[a\in\hD_{1}(w)]=\intfactsurv{t_2}+o(1).
$$
\end{enumerate}
\end{lemma}

To prove this lemma, 
we will need the asymptotic degree distribution of a variable node in $N_{\leq\ell}(w)$, which is a standard result.
\begin{proposition}\label{prop:distributionofnochildren}
 Let $w\in\cV\cup \cF$ and an integer $\ell=o(\log\log n)$ be given. Conditioned on the event $\eone$,
 for each $u \in D_{\le \ell}(w)\cap \cV$ independently, the number of children of $u$ in $N_{\leq\ell+1}(w)$
 is asymptotically distributed as $\Po(d)$.
\end{proposition}

We defer the proof to Appendix~\ref{app:offspringdist}. We can now prove Lemma~\ref{lem:survivalprobs}.
 
\begin{proof}[Proof of Lemma~\ref{lem:survivalprobs}]
We will prove both statements (1) and (2) by a common induction on $t$.
For $t=0$ the statements are clear since we have $\intvarsurv{0}=1$,
which corresponds to the fact that nothing ever gets deleted from $D_{\ell+1}(w)$,
while $\intfactsurv{1}=\prob(\Bi(r-1,1)\geq 1)=1$,
corresponding to the fact that internal factor nodes have $r-1 \ge 1$ children in the input graph
and therefore also no nodes of $D_\ell(w)$ will be deleted.

We assume both statements are true for~$t-1$ and aim to prove that they also hold for~$t$.
We first consider $u\in D_{\ell+1-2t}$ and let $X_u$ be the number of children of $u$ in $\hD_{\ell+2-2t}$.
Observe that the probability that $u$ survives in \coredetector\ is simply $\prob(X_u\geq 1)$.

By Proposition~\ref{prop:distributionofnochildren},
the number of children of $u$ is asymptotically $\Po(d)$, 
and by the induction hypothesis each child survives with probability $\intfactsurv{t}+o(1)$ independently of each other.
Therefore the asymptotic survival probability of $u$ is given by 
\[
\prob\left(\Bi(\Po(d),\intfactsurv{t}+o(1))\geq 1\right)=\prob\left(\Po(d(\intfactsurv{t}+o(1)))\geq 1\right)=\intvarsurv{t}+o(1),
\]
by definition of $\intvarsurv{t}$, as required for statement~(1).
Independence simply follows from the conditioning on $\eone$, which
in particular means that $N_{\le \ell+1}(w)$ is a tree.

Similarly for $a\in D_{\ell-2t}(w)$ we define $X_a$ to be the number of children of $a$ which survive.
Since $a$ is an internal factor node it has precisely $r-1$ children,
and by statement~1 which we have just proved, each child survives with probability $\intvarsurv{t}+o(1)$ independently.
Therefore the probability that $a$ survives the \coredetector\ is
$$
\Pr(\Bi(r-1,\intvarsurv{t}+o(1))\ge 1) = \intfactsurv{t+1}+o(1),
$$
as required for statement~(2). Again, independence simply follows from the conditioning on $\eone$.
 \end{proof}

A consequence of 
Lemma~\ref{lem:survivalprobs} is that, if $\ell$ is large,
the distribution of the number of children of the root $w$ which survive $\coredetectortwo$ is almost identical
to the claimed distributions in Theorem~\ref{thm:factor:reducedcoredegs}.
In order to quantify this,
for two discrete random variables $X,Y$ taking values in $\NN$ we define the \textit{total variation distance} as
\[
d_{\mathrm{TV}}(X,Y)\coloneqq\sum_{m\in\NN}\Big|\prob(X=m)-\prob(Y=m)\Big|.
\]

\begin{corollary}\label{cor:totalvariationdistance}
There exist $\eps=o(1)$ and $\ell=\ell(\eps)$ such that if we run $\coredetectortwo$ with input $\ell$ and root $w\in\cV\cup\cF$, then:
\begin{enumerate}
    \item If $w \in \cV$, then
    \[
    d_{\mathrm{TV}}\left(\td_1(w),\widetilde \Po(d\intfactsurvlim)\right) \le d_{\mathrm{TV}}\left(\hd_1(w), \Po(d\intfactsurvlim)\right)<\eps;
    \]
    \item If $w\in\cF$, then
    \[
    d_{\mathrm{TV}}\left(\td_1(w),\widetilde \Bi(r,\intvarsurvlim)\right)\le d_{\mathrm{TV}}\left(\hd_1(w),\Bi(r,\intvarsurvlim)\right)<\eps.
    \]
\end{enumerate}
\end{corollary}

Let us note that the second statement involves the $\Bi(r,\intvarsurvlim)$ distribution
rather than $\Bi(r-1,\intvarsurvlim)$ since $w$ is the root, and therefore all $r$ of its
neighbours are children rather than just $r-1$ children and one parent.

\begin{proof}
We will prove only the first statement; the proof of the second is very similar.

Note that the first inequality follows directly from the definitions of $\hd_1(w)$ and the $\widetilde\Po$ distributions.
There are $2$ reasons why the second inequality is not quite immediate from Lemma~\ref{lem:survivalprobs}:
\begin{enumerate}[label=\textnormal{\textbf{(R\arabic*)}}]
    \item\label{item:eventereasonone}Lemma~\ref{lem:survivalprobs} has the conditioning on the event $\eone$;
    
    \item\label{item:eventereasontwo} $\intfactsurv{(\ell+1)/2}$ in Lemma~\ref{lem:survivalprobs} has been replaced by $\intfactsurvlim$.

\end{enumerate}
To overcome \ref{item:eventereasonone},
for ease of notation we set $q_j\coloneqq \Pr\left(\hd_1(w)=j\right)$ and\linebreak[4] $q_j'\coloneqq\Pr\left(\hd_1(w)=j|\eone\right)$ for each $j\in \NN$,
and define
\begin{align*}
J^+ & \coloneqq\left\{j:q_j > q_j'\right\} \qquad  \text{and}\qquad 
J^-  \coloneqq\left\{j:q_j < q_j'\right\}.
\end{align*}
Since for any $w \in \cV \cup \cF$ we have
$$
\sum_{j=0}^\infty (q_j-q_j') = \sum_{j=0}^\infty q_j - \sum_{j=0}^\infty q_j' = 1-1=0,
$$
we deduce that
$$
\sum_{j=0}^\infty \left|q_j-q_j'\right|
 = \sum_{j \in J^+} \left(q_j-q_j'\right) - \sum_{j \in J^-}\left(q_j-q_j'\right)
= 2\sum_{j \in J^+}\left(q_j-q_j'\right).
$$
Therefore we have
\begin{align}
d_{\mathrm{TV}}\left(\hd_1(w),\hd_1(w)|_{\eone}\right)
& = 2\sum_{j\in J^+} \left(q_j-q_j'\right) \nonumber\\
& \leq 2\cdot \sum_{j \in J^+} \Pr\left(\hd_1(w)=j\right) - \Pr\left( (\hd_1(w)=j)\cap\eone\right) \nonumber\\
& \leq 2 \cdot\Pr\left(\overline{\eone}\right) \nonumber\\
& \leq 2 \exp\left(-\Theta\left(\sqrt{\log n}\right)\right) \nonumber\\
& \le \eps/3, \label{eq:totalvariationconditioning}
\end{align}
where we applied Lemma~\ref{lem:eventE} in the penultimate line, and where the last line holds
for $\eps$ tending to $0$ sufficiently slowly.

To address \ref{item:eventereasontwo} we note that if $\ell = \ell(\eps)$ is sufficiently large, then
\begin{equation}\label{eq:totalvariationpoissons}
    d_{\mathrm{TV}}\left(\Po(d\intfactsurv{(\ell+1)/2}),\Po(d\intfactsurvlim)\right)<\eps/3
\end{equation}
because $\intfactsurv{t} \xrightarrow{t\to \infty} \intfactsurvlim$ by definition.

To complete the proof, observe that Lemma~\ref{lem:survivalprobs}~\ref{eq:probfactnodesurvives} implies that
\begin{align*}
d_{\mathrm{TV}}\left(\hd_1(w)|_{\eone} \, , \, \Po(d\intfactsurv{(\ell+1)/2})\right) =
d_{\mathrm{TV}}\left(\hd_1(w)|_{\eone} \, , \, \Bi\left(\Po(d),\intfactsurv{(\ell+1)/2} \right)\right) \le \eps/3,    
\end{align*}
and combining this with~\eqref{eq:totalvariationconditioning} and~\eqref{eq:totalvariationpoissons},
we obtain
\begin{align*}
d_{\mathrm{TV}}\left(\hd_1(w),\Po(d\intfactsurvlim)\right)
& \le d_{\mathrm{TV}}\left(\hd_1(w),\hd_1(w)|_{\eone}\right)
	+ d_{\mathrm{TV}}\left(\hd_1(w)|_{\eone},\Po(d\intfactsurv{(\ell+1)/2})\right)\\
	& \hspace{3cm}+ d_{\mathrm{TV}}\left(\Po(d\intfactsurv{(\ell+1)/2},\Po(d\intfactsurvlim)\right)\\
	& \le \eps
\end{align*}
as required.
\end{proof}
As a further consequence of Corollary~\ref{cor:totalvariationdistance}, we can asymptotically determine the expected degree distribution in $G_{\ell}$.
\begin{corollary}\label{cor:expectationvarfac}
There exist $\eps = o(1)$ and $\ell=\ell(\eps)$ such that for all $j \in \NN$, 
\[\expec\left(\varpropl{j}{\ell}\right)=\prob\left(\widetilde\Po(d\intfactsurvlim)=j\right)\pm\eps,\]
and
\[\expec\left(\factpropl{j}{\ell}\right)=\prob\left(\widetilde\Bi(r,\intvarsurvlim)=j\right)\pm\eps.\]
\end{corollary}

\begin{proof}
Observe that for any $j \in \NN$ and any variable node $w$, by linearity of expectation and Corollary~\ref{cor:Gldegalmosteq},
for some $\eps_1=o(1)$ we have
$$
\expec\left(\varpropl{j}{\ell}\right)=\prob(d_{G_{\ell}}(w)=j) = 
\prob\left(\td_1(w)=j\right) \pm \eps_1
$$
Furthermore, Corollary~\ref{cor:totalvariationdistance} implies that for some $\eps_2=o(1)$ we have
$$
\prob\left(\td_1(w)=j\right) = \prob\left(\widetilde\Po(d\intfactsurvlim)=j\right)\pm\eps_2.
$$
Combining these two approximations and setting $\eps = \eps_1+\eps_2 = o(1)$ completes the proof of the first statement.
The second statement is proven similarly.
\end{proof}
We will show concentration of $\varpropl{j}{\ell}$ and $\factpropl{j}{\ell}$ around their expectations
by an application of an Azuma-Hoeffding inequality (Lemma~\ref{lem:Azuma}).
We will therefore define a vertex-exposure martingale
(or, using the terminology of factor graphs, a variable-exposure martingale)
 $X_0,\ldots,X_n$ on $G$, where
$X_0=\expec(\varpropl{j}{\ell})$ and $X_n=\varpropl{j}{\ell}$. This means that we reveal the variable nodes of $G$
and their incident edges one by one,
successively revealing more and more information about the input graph.
We would like to show that this martingale satisfies a Lipschitz property so that we can apply Lemma~\ref{lem:Azuma}.
However, this Lipschitz property cannot be guaranteed in general.
We therefore restrict attention to a class of factor graphs on which the Lipschitz property does indeed hold.

\begin{definition}\label{def:graphclass}
Let $\cG$ be the class of factor graphs $G=G(H)$ of $r$-uniform hypergraphs $H$ on vertex set $[n]$
such that no node has degree greater than $\log n$ in $G$.
\end{definition}
It is important to observe that for fixed $r$ and for $p$ in the range we are considering
(i.e. $\Theta\left(\binom{n-1}{r-1}^{-1}\right)$), the factor graph
 $G^r(n,p)$ is very likely to lie in $\cG$, so the restriction to this class is reasonable.
\begin{claim}\label{claim:maxdeg}
We have 
\[\prob(G^r(n,p)\in\cG)=1-o(1).\]
\end{claim}
The (standard) proof appears in Appendix~\ref{app:maxdeg}.

Now let $\barG=G^r(n,p)|_{\cG}$ be the factor graph of the $r$-uniform binomial random hypergraph $H^r(n,p)$
conditioned on being in $\cG$. For each $i\in\NN$ let $Z_i\in\{0,1\}^{\binom{[i-1]}{r-1}}$
be the sequence $(Y_{i,A})_{A \in \binom{[i-1]}{r-1}}$ of indicator random variables of the events
that there is a factor node of $\barG$ whose neighbours are $\{i\}\cup A$
(or equivalently that $\{i\}\cup A$ forms an edge of $H^r(n,p)$).
Let $f$ be a graph theoretic function.
Since there is a one-to-one correspondence between the factor graph $\barG$ and its associated sequence $(Z_i)=(Z_i)_{i \in [n]}$,
we can view the 
random variable $X\coloneqq f(\barG)$ as a function of the sequence $(Z_i)$ and we let $X_i\coloneqq\expec(X|Z_1,\ldots,Z_i).$
This martingale can be seen as revealing variable nodes of $\barG$ one by one, and $X_i$ is
the expected value of $X$ conditioned on the information revealed after $i$ steps.

\begin{definition}\label{lem:lipschitzproperty}
A graph function $f$ is \emph{$c$-variable-Lipschitz} if whenever two factor graphs $G$ and $G'$ differ at exactly one variable node
(but any number of factor nodes),
then $|f(G)-f(G')|\leq c$ holds.
\end{definition}

This is very similar to the standard notion of being $c$-vertex-Lipschitz -- indeed, it corresponds
precisely to the original hypergraph being $c$-vertex-Lipschitz -- and it is
a standard observation that being $c$-variable-Lipschitz
implies that the corresponding variable-exposure martingale satisfies $|X_i-X_{i-1}|\leq c$,
see e.g.\ Corollary 2.27 in~\cite{JansonLuczakRucinskiBook}.

For the purposes of the Azuma-Hoeffding argument,
we will view $\varpropl{j}{\ell},\factpropl{j}{\ell}$ as graph functions---thus by the original
definition we have e.g.\ $\varpropl{j}{\ell} = \varpropl{j}{\ell}(G^r(n,p))$,
i.e.\ the graph function applied to the factor graph of the $r$-uniform binomial random hypergraph.

 \begin{lemma}
For $\ell\in\NN$, the graph functions $\varpropl{j}{\ell},\factpropl{j}{\ell}$ are $c_{\ell}$-variable-Lipschitz within the class $\cG$, where $c_{\ell}=\frac{2(\log n)^{\ell+1}}{n}$.
 \end{lemma}
 
 \begin{proof}
 We will show that $\varpropl{j}{\ell}$ is $c_\ell$-variable-Lipschitz---the corresponding proof for $\factpropl{j}{\ell}$ is completely analogous.
 
 Suppose that $G,G' \in \cG$ are two factor graphs which differ at exactly one variable node, say $v$,
 and let $A_j = A_j(G,\ell)$ and $A_j'=A_j'(G',\ell)$ be the sets of variable nodes of degree~$j$ in $G_\ell$ and $G_{\ell}'$, respectively.
 Since $G,G'$ differ only at $v$, any edges in the symmetric difference $G_\ell \Delta G_{\ell}'$ must lie within
 distance $\ell+1$ of $v$ in either $G$ or $G'$,
 and therefore any nodes in the symmetric difference $A_j \Delta A_j'$ must lie within distance $\ell+1$ of $v$ in either $G$ or $G'$.
Since $G,G' \in \cG$, and therefore have maximum degree at most $\log n$,
 there are at most $(\log n)^{\ell+1}$ such nodes in each of $G,G'$, and
therefore $|A_j \Delta A_j'| \le 2(\log n)^{\ell+1}$. Since $\varpropl{j}{\ell}(G)=|A_j|/n$ and $\varpropl{j}{\ell}(G')=|A_j'|/n$,
we deduce that
$$
\left|\varpropl{j}{\ell}(G)-\varpropl{j}{\ell}(G')\right| \le \frac{2(\log n)^{\ell+1}}{n}
$$
as required.
 \end{proof}
 
Now for $j,\ell \in \NN$, let us define the events $B_j=B_j(\ell) \coloneqq \left\{\left|\varpropl{j}{\ell} - \expec\left(\varpropl{j}{\ell}\right)\right| \ge n^{-1/3}\right\}$
and $\hat B_j=\hat B_j(\ell) \coloneqq \left\{\left|\factpropl{j}{\ell} - \expec\left(\factpropl{j}{\ell}\right)\right| \ge n^{-1/3}\right\}$.
 
\begin{lemma}\label{lem:concentration}
Let $G=G^r(n,p)$ be a random factor graph and $\ell = o(\log \log n)$. Then
\[
\Pr\left(\bigcup_{j \in \NN}\left(B_j \cup \hat B_j\right) \right) = o(1).
\]
\end{lemma} 

\begin{proof}
For convenience, we will prove that $\Pr(\bigcup_{j \in \NN}B_j)=o(1)$---the proof of the corresponding
statement for the $\hat B_j$ is almost identical.

Let us define $B\coloneqq \bigcup_{j\in \NN}B_j$, and observe that
\begin{align}\label{eq:azumaconditioning}
\Pr(B) & = \Pr(B \cap \{G \in \cG\}) + \Pr(B \cap \{G \notin \cG\}) \nonumber \\
& \le \Pr(B | G \in \cG) + \Pr(G \notin \cG) \nonumber \\
& \le \sum_{j \in \NN} \Pr(B_j | G \in \cG) + o(1),
\end{align}
where the last line follows by a union bound and Claim~\ref{claim:maxdeg}.

Now in order to bound the sum, let us first fix $j \in \NN$.
Since $\ell=o(\log\log n)$, we have 
$(\log n)^ {\ell+1}=n^ {o(1)}$,
and by Lemma~\ref{lem:Azuma} with $c_{\ell}=\frac{2(\log n)^ {\ell+1}}{n}$ and $s=n^ {-1/3}$,
we have
\begin{align*}
    \prob\left(\left|\varpropl{j}{\ell}-\expec\left(\varpropl{j}{\ell}\right)\right|\geq n^{-1/3} \; \Bigg| \; G \in \cG\right)
    &\leq 2\cdot \exp\left(-\frac{n^{-2/3}}{2\cdot 2(\log n)^{\ell+1}/n}\right)< \exp\left(-n^{1/2}\right).
\end{align*}
Observing that the degree of any node is (deterministically) bounded by $\binom{n}{r-1}$,
we have
$$
\sum_{j\in \NN}\Pr(B_j | G \in \cG) \le \left(\binom{n}{r-1}+1\right)\exp(-n^{1/2}) = o(1).
$$
Substituting this into~\eqref{eq:azumaconditioning} completes the proof.
\end{proof}
We are now able to give the proof of Lemma~\ref{lem:mainlemma1}.
\begin{proof}[Proof of Lemma~\ref{lem:mainlemma1}]
By Corollary~\ref{cor:expectationvarfac} (with $\eps/2$ in place of $\eps$),
if $\ell = \ell(\eps)$ is sufficiently large  we have
\[
\expec\left(\varpropl{j}{\ell}\right)=\prob\left(\widetilde\Po\left(d\intfactsurvlim\right)=j\right)\pm\eps/2
\]
for any $j \in \NN$.
Furthermore, by
Lemma~\ref{lem:concentration}, we have that \whp\ for all $j\in\NN$
\[
\varpropl{j}{\ell}=\expec\left(\varpropl{j}{\ell}\right)+o(1) = \expec\left(\varpropl{j}{\ell}\right) \pm \eps/2
\]
for $\eps$ tending to $0$ sufficiently slowly,
and combining these two facts proves the lemma for variable nodes. The proof for factor nodes is essentially identical.
\end{proof}

\section{Subcriticality: Proof of Lemma~\ref{lem:mainlemma2}}\label{sec:prooflowerbound}
Our goal in this section is to show that after some large number $\ell$ rounds of the peeling process on $G^r(n,p)$ have been completed,
\whp\ very few nodes will be disabled in subsequent rounds (at most $\eps n$ for some $\eps=\eps(n) =o(1)$),
thus proving Lemma~\ref{lem:mainlemma2}.

\emph{
Let us fix $\ell,\eps_1$ as in Lemma~\ref{lem:mainlemma1}, and for the rest of this section we will assume that
the high probability events of Lemma~\ref{lem:mainlemma1} and Proposition~\ref{prop:numberfactornodes}
both hold.
}

To help intuitively describe the argument, let us suppose for simplicity that in round $\ell$ exactly \emph{one}
node $x_0$ is disabled and we consider the future effects of such a disabling.
Since $x_0$ was disabled, it had degree at most one,
and therefore there is at most one neighbour $x_1$ whose degree is decreased as a result.
If $x_1$ originally had degree two, it now has degree one and will therefore be disabled in round $\ell+1$.
Continuing in this way, we observe that we will never be disabling more than one node in any subsequent round.
Furthermore, if we reach a node $x_i$ whose original degree was not exactly two, the peeling process stops
(either without disabling this node, or once it has been disabled and no further nodes' degrees are decreased).
Heuristically, it will not take long before we reach a node whose original degree was not exactly two---this
is because Lemma~\ref{lem:mainlemma1} implies in particular that a constant proportion
of the nodes have degree at least three.

Of course, in reality there may be more than one node disabled in round $\ell$.
This slightly complicates matters because some node may receive the knock-on effects
of more than one disabling in round $\ell$, and therefore have its degree decreased by more than one.
However, this will turn out to be no more than a technical nuisance.

We first need a result which states that almost any graph with a fixed (reasonable) degree sequence
is approximately equally likely to be $G_\ell$, the graph obtained from $G=G^r(n,p)$ after $\ell$ rounds of the peeling process.
To introduce this result, we need some definitions.

\begin{definition}
An \emph{$r$-duplicate} in a factor graph consists of two factor nodes and $r$ variable nodes which together
form a copy of $K_{2,r}$.
\end{definition}

Observe that an $r$-duplicate would correspond to a double-edge in an $r$-uniform hypergraph,
which in our model cannot occur since the hypergraph must be simple.
Therefore our factor graphs may not contain any $r$-duplicates.
On the other hand, a ``loop'', in the sense of an edge which contains the same vertex more than once,
must involve a double-edge in the corresponding factor graph. This motivates the following definition,
which (roughly) describes when a factor graph corresponds to a simple hypergraph.

\begin{definition}\label{def:rplain}
We say that a factor graph is \emph{$r$-plain} if:
\begin{enumerate}
\item it contains no double-edge, i.e.\ two edges between the same variable node and factor node;
\item it contains no $r$-duplicates.
\end{enumerate}
\end{definition}

\begin{claim}\label{claim:uniformity}
Suppose $H_1,H_2$ are two $r$-plain factor graphs with common variable node set $\cV=\cV(H_1)=\cV(H_2)=[n]$
and with factor node sets $\cF_1=\cF(H_1)$ and $\cF_2=\cF(H_2)$.
Suppose further that there is a bijection $\phi: \cV \cup \cF(H_1) \to \cV \cup \cF(H_2)$
such that
\begin{itemize}
\item $\phi(\cV)=\cV$;
\item $d_{H_2}(\phi(v))=d_{H_1}(v)$ for all $v \in \cV \cup \cF(H_1)$.
\end{itemize}
Let $G=G^r(n,p)$.
Then 
\[
\prob\left(G_{\ell}=H_1\right)=\prob\left(G_{\ell}=H_2\right).
\]
\end{claim} 

This claim is very similar to standard results for simple graphs or hypergraphs
(see e.g.~\cite{Molloy2005})
and we defer the proof to Appendix~\ref{sec:proofofuniformity}. However, we note that in our setting there 
is one subtle technical difficulty to overcome which does not appear in many other cases,
namely that given a factor graph $G$ such that $G_\ell=H_1$,
if we transform $G$ by changing $H_1$ to $H_2$ but otherwise leaving $G$ unchanged,
we need to show that
the resulting graph is indeed the factor graph of an $r$-uniform hypergraph, and in particular is $r$-plain.

Claim~\ref{claim:uniformity} tells us that the factor graph $G_\ell$ after $\ell$ rounds
of the peeling process is uniformly random conditioned on its degree sequence and being $r$-plain.
Since Lemma~\ref{lem:mainlemma1} tells us the degree sequence quite precisely,
this is very helpful. We can change our point of view by saying that we first reveal
the degree sequence of $G_\ell$ without revealing any of its edges,
and subsequently we reveal edges only as required. More precisely, we consider the
\emph{configuration model}, in which each node is given half-edges based on its degree,
and we generate a uniformly random perfect matching between the two classes of half-edges
(at variable and factor nodes) conditioned on the
resulting factor graph being $r$-plain.
We need to know that this conditioning is not too restrictive, i.e.\ that the probability
that the resulting factor graph is $r$-plain is not too small. This will be stated
in Proposition~\ref{prop:simpleprob}, for which we first need some preliminaries.
We begin by observing that
$G^r(n,p)$ does not have too many nodes of high degree.

\begin{definition}\label{def:largedegsquare}
Given a function $\omega=\omega(n) \to \infty$, we say that a factor graph $H$
has property $\largedegsquare=\largedegsquare(\omega,n)$
if
$$
\sum_{\substack{v \in \cV(H):\  d(v)>\omega}}\  d(v)^2 = o(n).
$$
Furthermore given $\eps>0$ we say that $H$ has property
$\gooddegs=\gooddegs(\eps,\omega,n)$ if it satisfies property $\largedegsquare(\omega,n)$
and also satisfies the conclusion of Lemma~\ref{lem:mainlemma1} (with this $\eps$).
\end{definition}

\begin{claim}\label{claim:largedegsquare}
For any $\omega\xrightarrow{n \to \infty} \infty$,
with high probability $G^r(n,p)$ has property $\largedegsquare(\omega,n)$.
\end{claim}
\begin{proof}
For $k \in \NN$, let us define $X_k$ to be the number of variable nodes of degree $k$
and $X_{\ge k}\coloneqq \sum_{j \in \NN_{\ge k}} X_j$.
Observe that the expected degree of a vertex is $\binom{n-1}{r-1}p = (1+o(1))d$.
It is a standard fact that the degrees of vertices are approximately
distributed as independent $\Po(d)$ variables. More formally (though much weaker),
it is an easy exercise in the second-moment method
to prove that whp, for any $k\in \NN$ we have $X_{\ge k} \le n\cdot \Pr(\Po(2d)\ge k)$ (we omit the details).
We therefore have
\begin{align*}
\sum_{v \in \cV(H): d(v) \ge \omega} d(v)^2 = \sum_{k \ge \omega} k^2 X_k
& \le n\sum_{k \ge \omega} k^2 \frac{e^{-2d}(2d)^k}{k!} = n \cdot (1+o(1))\omega^2\frac{e^{-2d}(2d)^{\omega}}{\omega!} = o(n),
\end{align*}
as required.
\end{proof}

Note that if $\largedegsquare$ holds in a factor graph $G$, then it also holds
in any subgraph of $G$, and in particular in $G_{\ell}$, the factor graph obtained
after $\ell$ steps of the peeling process. Together with Lemma~\ref{lem:mainlemma1},
this shows that, with $\ell$ and $\eps$ as in given in that lemma and any $\omega \to \infty$,
setting $G = G^r(n,p)$,
with high probability $G_\ell$ satisfies $\gooddegs(\eps,\omega,n)$.

Let us observe further that $\gooddegs$ is a property that depends only on the
degree sequences of variable and factor nodes of the graph, and therefore with a slight
abuse of terminology we may also say that it holds in a factor graph with half-edges,
where we have not yet determined which half-edges will be matched together.

\begin{proposition}\label{prop:simpleprob}
Let $\eps = o(1)$ and
suppose that $n$ variable nodes and $m=(1+o(1))\frac{dn}{r}$ factor nodes are given half-edges
in such a way that property $\gooddegs(\eps,\eps^{-1/4},n)$ holds.
Suppose also that the total numbers of half-edges at factor nodes and at variable nodes are equal,
and that we construct a uniformly random perfect matching between these two sets of half-edges.
Then there exists a constant $c_0>0$ (independent of $\eps,n$)
such that for sufficiently large $n$ the probability that the resulting factor graph is
$r$-plain is at least $c_0$.
\end{proposition}

The proof of Proposition~\ref{prop:simpleprob} is a standard exercise
in applying the method of moments to determine the asymptotic distribution
of the number of double edges and $r$-duplicates -- we omit the details.
The proposition states in particular that we may condition on the
resulting graph being $r$-plain without skewing the distribution of the matching too much.
More precisely, any statements that are true with high probability for the uniformly random perfect matching
are also true with high probability under the condition that the resulting graph is simple.
Therefore in what follows, for simplicity we will suppress this conditioning.

\begin{definition}[Change process]\label{def:changeprocess}
We will track the changes that the peeling process makes after reaching round $\ell$
by revealing information a little at a time as follows.
\begin{itemize}
\item Reveal the degrees of all nodes.
\item While there are still nodes of degree one, pick one such node $x_0$.
\begin{itemize}
\item Reveal its neighbour $x_1$, delete the edge $x_0x_1$ and update the degrees of $x_0,x_1$.
\item If $x_1$ now has degree one, continue from $x_1$; otherwise find a new $x_0$ (if there is one).
\end{itemize}
\end{itemize}
\end{definition}

\noindent In other words, we track the changes in a depth-first search manner (rather than the breadth-first
view of considering rounds of the peeling process). We call this the \emph{change process}.

Observe that we only reveal edges one at a time (just before deleting them).
The following claim is simple given Lemma~\ref{lem:mainlemma1}, but is the essential heart of our proof.
Recall that $\eps_2 \coloneqq \sqrt{\eps_1}$ as defined in Lemma~\ref{lem:mainlemma2}.

\begin{claim}\label{claim:terminationprob}
Let $G'$ be any graph obtained from $G_\ell$ by deleting at most $\eps_2 n$ edges.
Then when revealing the second endpoint of any half-edge, the probability of revealing a node of degree at least three
is at least
$$
\min\left\{\frac{(d\intfactsurvlim)^2 \exp(-d\intfactsurvlim)}{2} , \frac{(r-1)(r-2)\intvarsurvlim^2(1-\intvarsurvlim)^{r-3}}{2}\right\} - 20\eps_2.
$$
In particular there exists a constant $c=c(r,d)>0$ such that this probability is at least $c$.
\end{claim}
We defer the proof to Appendix~\ref{sec:proofofterminationprob}.

\noindent This claim tells us that, provided we have not deleted too many edges so far, there is a reasonable probability
of revealing a node of degree at least three, which blocks the continued propagation of any deletions.

Let $c=c(r,d)$ be as in Claim~\ref{claim:terminationprob}
and let us set $\delta_1\coloneqq\eps_1^{3/4}$.
We now define an abstract branching process
which will provide an upper coupling on the change process starting from $G_\ell$.

\begin{definition}
Let $\cT$ be a branching process which begins with $\delta_1n$ vertices in generation~$0$,
and in which each vertex independently has a child with probability $1-c$, and otherwise has no children.
\end{definition}

\begin{proposition}\label{prop:changecoupling}
The process $\cT$ can be coupled with the change process in such a way that,
if both processes are run until one of the stopping conditions
\begin{itemize}
\item $\cT$ has reached size at least $\eps_2 n$;
\item $\cT$ has died out,
\end{itemize}
is satisfied, then $\cT$ forms an upper coupling on the change process.
\end{proposition}

\begin{proof}
We first need to show that \whp\ we make at most $\delta_1 n$ changes in round $\ell+1$ of the peeling process.
Since we have assumed that the high probability statement of Lemma~\ref{lem:mainlemma1} holds, we have
$\varpropl{j}{\ell}=\prob(\widetilde\Po(d\intfactsurvlim)=j)\pm\eps_1$ and $\factpropl{j}{\ell}=\prob(\widetilde\Bi(r,\intvarsurvlim)=j)\pm\eps_1$.
By the definition of the peeling process (Definition~\ref{def:peeling}) the only change we make when moving from $G_{\ell}$ to $G_{\ell+1}$ is that we disable all nodes of degree one.
The proportion of such variable and factor nodes in $G_{\ell}$ is $\varpropl{1}{\ell}$ and $\factpropl{1}{\ell}$ respectively.
Recalling that $\prob(\widetilde\Po(d\intfactsurvlim)=1) = \prob(\widetilde\Bi(r,\intvarsurvlim)=1) =0$,
this immediately implies that at most
$\eps_1(m+n) \le \delta_1 n$ nodes are disabled in round $\ell+1$ of the peeling process,
and these disablings represent the first nodes of the change process.

Now the proposition follows directly from the observation that in the change process,
a node only has at most one incident edge deleted (if it has degree one), and
therefore at most one neighbour is revealed, along with Claim~\ref{claim:terminationprob},
which implies that the probability of not causing any further changes is at least $c$.
The first stopping condition ensures that we have deleted at most $\eps_2 n$ edges,
and therefore the assumptions of Claim~\ref{claim:terminationprob} are indeed satisfied.
\end{proof}

In view of Proposition~\ref{prop:changecoupling}, it is enough to prove that \whp\ the
branching process $\cT$ dies out (i.e.\ fulfills the second stopping condition)
before reaching size $\eps_2 n$.

\begin{proposition}\label{prop:changebranchsmall}
\Whp\ $\cT$ contains at most $\eps_2 n$ vertices.
\end{proposition}

\begin{proof}
In order to reach size $\eps_2 n$, the first (at most) $\eps_2 n$ vertices of the process
would have to have a total of at least $(\eps_2-\delta_1)n$ children,
which, by Lemma~\ref{lem:chernoffbounds}, occurs with probability at most
\begin{align*}
\Pr\left(\Bi(\eps_2 n,1-c)\ge (\eps_2-\delta_1)n\right) & \le 2\cdot \exp\left(-\frac{(c\eps_2-\delta_1)^2 n^2}{2\eps_2 n + 2(c\eps_2-\delta_1)n/3}\right)\\
& \le 2\cdot \exp\left(-\frac{(c\eps_2/2)^2n}{3\eps_2}\right) =o(1)
\end{align*}
(since $\eps_2 n = \sqrt{\eps_1}n \to \infty$)
as required.
\end{proof}

We can now complete the proof of Lemma~\ref{lem:mainlemma2}.

\begin{proof}[Proof of Lemma~\ref{lem:mainlemma2}]
Proposition~\ref{prop:changecoupling} implies that the probability that at least $\eps_2 n$
further nodes are disabled after round $\ell$ in the peeling process is at most the probability
that $\cT$ reaches size $\eps_2 n$. But this event has probability $o(1)$ by
Proposition~\ref{prop:changebranchsmall}.

This proves that \whp\ at most $\eps_2 n$ further nodes are disabled
after round $\ell$ in the peeling process. To show the remaining two statements,
observe that
since disabling a node deletes at most one edge, and therefore changes the degree of at most one variable node
and at most one factor node, at most $\eps_2 n$ nodes of each type will have their degree changed.
It follows immediately that for any $j \in \NN$ we have $\varprop{j} = \varpropl{j}{\ell} \pm \eps_2$,
while similarly
$$
\factprop{j} = \factpropl{j}{\ell} \pm \frac{\eps_2 n}{m} = \factpropl{j}{\ell} \pm \frac{2\eps_2 r}{d},
$$
where we have used the fact that \whp\ $\frac{n}{m} \le \frac{2r}{d}$ by Proposition~\ref{prop:numberfactornodes}.
\end{proof}

\section{Concluding remarks}\label{sec:concluding}

\subsection{Upper bound on $L_P,L_C$}
Theorem~\ref{thm:mainresultcycle} and Corollary~\ref{cor:upperboundlongestpath} state that \whp\ $L_P$ and $L_C$,
the length of the longest loose path and longest loose cycle respectively in $\cH$,
satisfy $L_P,L_C\leq(\min\{\oldeta,\gamma\}+o(1))\cdot n$, but which of $\oldeta,\gamma$ is smaller?
Recall that both $\oldeta$ and $\gamma$ are functions of $d$ and $r$, with $\oldeta=\gamma =0$ for $d<d^*$.
Numerical approximations and plots with Mathematica suggest that 
for $r=2,3$, we have
$\oldeta \le \gamma$
for all $d$,
but that for $r\ge 4$, for values of $d$ not too much larger than $d^*$ we have $\gamma<\oldeta$,
and there is a crossing point after which (for larger $d$) we have $\oldeta <\gamma$.
It would be interesting to investigate this behaviour more closely to determine whether this is indeed true,
and to determine the crossing point precisely as a function of $r$.

\subsection{High-order cores}
The methods used in this paper are amenable to more general definitions of cores in hypergraphs.
More precisely, vertex degrees are far from the only type of degrees that have been
extensively studied in hypergraphs---one can consider the degrees of $j$-sets of vertices for each $1\le j \le r-1$,
and for each $j$ there is a natural associated definition of a core. So far only the case $j=1$ has been studied,
but it would also be interesting to consider ``high-order cores'', i.e.\ $j\ge 2$.

\bibliographystyle{plain}
\bibliography{References}
\newpage
\appendix

\section{Analysis of fixed-point equation}\label{app:analysisfixedpoint}
Recall the fixed-point equation~\eqref{eq:fixedpointequation}
\[
1-\rho=F(1-\rho)
\]
with largest solution $\intvarsurvlim.$ Note that this equation has no solutions for $\rho\geq 1$ and that $0$ is a solution,
hence $0\leq\intvarsurvlim<1$. Our goal in this section is to prove Claim~\ref{claim:behaviouroffixedpointsol} and Lemma~\ref{lem:continuity}.
It will be more convenient to use a transformed equation: by substituting $x=1-\rho$ and taking logarithms on both sides (which is permissible
since both sides are positive in any solution)
we get the equivalent equation $\log x=-d(1-x^{r-1})$, or equivalently 
\begin{equation}\label{eq:definitionoff}
f(x)\coloneqq\log x+d(1-x^{r-1}) = 0,
\end{equation}
for $0<x\leq 1$. Furthermore we define $\invintvarsurvlim$ to be the smallest solution of~\eqref{eq:definitionoff}
and note that $\invintvarsurvlim=1-\intvarsurvlim$ holds. We now restate Claim~\ref{claim:behaviouroffixedpointsol}
in this new setting.

\begin{claim}\label{claim:behaviourfixedpointsolreformulated}\mbox{ }
\begin{enumerate}[label=\textnormal{\textbf{(F\arabic*)'}}]
\item\label{item:fixedpointsubcritical}If $d<d^*$, then $\invintvarsurvlim=1$.
\item\label{item:fixedpointsupercritical} If $d>d^*$, then $f(1)=0$ and there is a unique solution to~\eqref{eq:definitionoff} in $(0,1)$.
\end{enumerate}
\end{claim}

\begin{proof}[Proof of Claim~\ref{claim:behaviouroffixedpointsol}]
It is clear that Claim~\ref{claim:behaviouroffixedpointsol} follows directly from Claim~\ref{claim:behaviourfixedpointsolreformulated}
since $\intvarsurvlim = 1-\invintvarsurvlim$.
\end{proof}

To prove Claim~\ref{claim:behaviourfixedpointsolreformulated}, we define
\begin{equation*}
f_n(x)\coloneqq\log x+d\left(1+\frac{1}{\omega}\right)(1-x^{r-1}),
\end{equation*}
where $\omega=\omega(n)$ is a function with $\omega\xrightarrow{n\to\infty}\infty$.
Observe that $f(x)=\lim\limits_{n\to\infty}f_n(x)$ for each $x$.
We would like to treat $f_n$ and $f$ simultaneously, therefore with a slight abuse of notation
we define $f_\infty = f$, and set $1/\omega=0$ for $n=\infty$.
Furthermore, whenever we use statements such as ``for $n$ sufficiently large'', this also includes $n=\infty$.

Now observe that
\begin{align*}
f'_n(x) &= \frac{1}{x}-d\left(1+\frac{1}{\omega}\right)(r-1)x^{r-2}, \\
f''_n(x) &= -\frac{1}{x^2}-d\left(1+\frac{1}{\omega}\right)(r-1)(r-2)x^{r-3}.
\end{align*}
The following fact collects properties of $f_n$ which are trivial to check and that will be used later in this section. 

\begin{fact}\label{fact:basicpropertiesf} For sufficiently large $n$ we have
\begin{enumerate}[label=\textnormal{\textbf{(P\arabic*)}}]
    \item\label{item:fnatone} $f_n(1)=0$.
    \vspace{0.3cm}
    \item\label{item:fntominusinfinity} $f_n(x)\xrightarrow{x\to 0}-\infty$.
    \vspace{0.3cm}
    \item\label{item:subcriticalderivative} If $d<d^*$, then $f'_n(x)>0$ for all $0<x<1$.
    \vspace{0.3cm}
    \item\label{item:supercriticalderivative} If $d>d^*$, then $f'_n(1)<0$.
    \vspace{0.3cm}
    \item\label{item:fnsecondderivative} $f''_n(x)<0$ for all $0<x\leq 1$.
\end{enumerate}
\end{fact}

\begin{proof}[Proof of Claim~\ref{claim:behaviourfixedpointsolreformulated}]
First we show \ref{item:fixedpointsubcritical}. Let $d<d^ *$. By~\ref{item:fnatone} and~\ref{item:subcriticalderivative} we have $f(x)<0$ for all $0<x<1$ and hence that $x=1$ is the unique solution of $f(x)=0$, i.e. $\invintvarsurvlim=1$. 

To see why~\ref{item:fixedpointsupercritical} holds we observe that $\ref{item:fnatone}$ and $\ref{item:supercriticalderivative}$ imply that there exists
$x_0$ with $0<x_0<1$ such that $f(x_0)>0$. Since~\ref{item:fntominusinfinity} holds by the intermediate value theorem there is $y_0$ with $0<y_0<x_0$ such that $f(y_0)=0$, which implies that $\invintvarsurvlim<1$.

To prove uniqueness, suppose that there are two solutions $x_1,x_2$ of~\eqref{eq:definitionoff} with $0<x_1<x_2<1$. Since $1$ is also a solution,
by the mean value theorem there exist $y_1,y_2$ with $x_1<y_1<x_2<y_2<1$ such that
$f'(y_1)=f'(y_2) = 0$. Then by the mean value theorem applied to $f'$, there exists a $z_1$ with $y_1<z_1<y_2$ such that
$f''(z_1)=0$. However, this contradicts~\ref{item:fnsecondderivative}.
\end{proof}

We need one last fact before we can give the proof of Lemma~\ref{lem:continuity}.
\begin{claim}\label{claim:derivativebigbeforefixedpoint}
If $d>d^*$, then there exists $\delta>0$ such that for all sufficiently large $n$ and all $0< x\leq\invintvarsurvlim$, we have
\begin{equation}\label{eq:fnderivativeestimate}
f'_n(x)>\delta.
\end{equation} 
\end{claim}
\begin{proof}
Observe that $f'(\invintvarsurvlim) >0$,
which follows from the fact that $\invintvarsurvlim <1$, the fact that $f(1)=0$ and~\ref{item:fnsecondderivative}.
Now $f_n'(\invintvarsurvlim) \xrightarrow{n\to \infty} f'(\invintvarsurvlim)$,
and therefore there
exists $\delta>0$ and $N\in\NN$ such that for all $n\geq N$ we have $f'_n(\invintvarsurvlim)>\delta$.
Together with~\ref{item:fnsecondderivative}, the claim follows since $f'_n(x)>f'_n(\invintvarsurvlim) >\delta$
holds for all $0<x<\invintvarsurvlim$ and $n \ge N$. 
\end{proof}

With these considerations we are able to give the proof of Lemma~\ref{lem:continuity}.

\begin{proof}[Proof of Lemma~\ref{lem:continuity}]
Let $Q_n$ denote the set of zeros of $f_n(x)$ and let $q_n\coloneqq\invintvarsurvlim\left(r,d\left(1+\frac{1}{\omega}\right)\right)$. Observe that $q_n=\min_n Q_n$ and let $q_*\coloneqq\lim\limits_{n\to\infty}q_n$. Our goal is to show that
\[
q_*=\invintvarsurvlim(r,d)=:\invintvarsurvlim,
\]
or in other words, 
\begin{equation}\label{eq:qstarisrhostar}
\intvarsurvlim\left(r,d\left(1+\frac{1}{\omega}\right)\right)=\intvarsurvlim+o(1).
\end{equation} 
For all $n\in\NN$ we have $q_n\leq 1$ since $x=1$ is a solution to $f_n(x)=0$, and $q_n\geq 0$ by~\ref{item:fntominusinfinity}.
By Claim~\ref{claim:behaviourfixedpointsolreformulated} there exists some $N\in\NN$ such that for all $n\geq N$ the function $f_n$ has exactly one solution $\nu_n$, and since $f_n(\invintvarsurvlim)>0$ we have that $\nu_n<\invintvarsurvlim$. By the mean value theorem there exists $y_n$ with $\nu_n<y_n<\invintvarsurvlim$ such that 
\begin{equation*}
  f'_n(y_n)=\frac{f_n(\invintvarsurvlim)-f_n(\nu_n)}{\invintvarsurvlim-\nu_n}.
\end{equation*}
By~\eqref{eq:fnderivativeestimate} we deduce that
\[
\invintvarsurvlim-\nu_n=\frac{f_n(\invintvarsurvlim)-f_n(\nu_n)}{ f'_n(y_n)}\numleq{\eqref{eq:fnderivativeestimate}} \frac{f_n(\invintvarsurvlim)}{\delta}
\]
and hence  $\nu_n\xrightarrow{n\to\infty}\invintvarsurvlim$ (since $f_n(\invintvarsurvlim)\to 0$).

We have now proved that~\eqref{eq:qstarisrhostar} holds, and~\eqref{eq:relatingtherhos}
shows that then we also have
\[
\intfactsurvlim\left(r,d\left(1+\frac{1}{\omega}\right)\right)=\intfactsurvlim(r,d)+o(1).
\]
Furthermore $\oldeta$ and $\gamma$ as defined in~\eqref{eq:definitionofeta} and~\eqref{eq:definitionofgamma} are continuous functions in $\intvarsurvlim$ and $\intfactsurvlim$, respectively, both of which are themselves continuous functions in $d$ and the statement of Lemma~\ref{lem:continuity} follows.
\end{proof}

\section{Probabilistic Lemmas}\label{app:problemmas}

In this appendix we include various proofs of probabilistic results which were omitted from the main text.
These proofs are all standard applications of common techniques.

\subsection{Number of factor nodes}\label{app:numberfactornodes}

\begin{proof}[Proof of Proposition~\ref{prop:numberfactornodes}]
Let us set $a\coloneqq\binom{n}{r}p - \frac{dn}{r}$, and observe that
$$
a = (1+o(1))d \frac{\binom{n}{r}}{\binom{n}{r-1}} - \frac{dn}{r} = o(1)\frac{dn}{r} = o(n).
$$
We will prove the proposition for any $\omega_0$ that satisfies $\omega_0 = o(n/a)$
and $\omega_0 = o(n^{1/4})$, so let us fix an $\omega_0 \to \infty$ that satisfies these properties.

The number of factor nodes in $\randfact$ is simply the number of edges in $H^r(n,p)$, which is distributed as
$X\sim\Bi\big(\binom{n}{r},p\big)$. This in turn has expectation
\begin{equation}\label{eq:expectationedges}
    \binom{n}{r}p = (1+o(1))\frac{n^r}{r!} \cdot \frac{d(r-1)!}{n^{r-1}} = (1+o(1))\frac{dn}{r}.
\end{equation}
By Lemma~\ref{lem:chernoffbounds} with $N=\binom{n}{r}$ and $s=n^{3/4}$ we have
\[\prob\left(|X-\expec(X)|\geq s\right)\leq 2\cdot\exp\left(-\frac{n^{1/2}r}{3d}\right)=o(1),\]
since 
\[2\left(Np+\frac{s}{3}\right)\numeq{\eqref{eq:expectationedges}}(1+o(1))\frac{2dn}{r}\leq\frac{3dn}{r}.\]
Hence, \whp\ the number of edges $m$ satisfies
$$
\left|m-\frac{dn}{r}\right| \le \left|m-\binom{n}{r}p\right|  + \left|\binom{n}{r}p - \frac{dn}{r}\right|
\le n^{3/4} + \left|a\right| = o\left(\frac{n}{\omega_0}\right)
\le \frac{1}{\omega_0}\cdot \frac{dn}{r},
$$
as required.
\end{proof}

\subsection{Sprinkling}\label{app:sprinkling}

\begin{proof}[Proof of Lemma~\ref{lem:standardsprinkling}]
First let $\omega'$ be a function which tends to infinity arbitrarily slowly, and in particular such that $\frac{p_1 n^r}{(\omega')^r\omega} \to \infty$.
Observe that we may assume that $L_P^{(1)} \ge n/\omega'$, since otherwise the trivial bound $L_C^{(2)}\ge 0$
is sufficient.

Since $H_1,H_2$ are coupled such that $H_1\subset H_2$, we have $H_2 \sim H_1 \cup H_0$ for
a random hypergraph $H_0 \sim H^r(n,p_0)$ which is independent of $H_1$ and where $p_0 = \frac{p_2-p_1}{1-p_1} \ge p_1/\omega$.

Let $P_0'$ be a longest loose path in $H_1$.
Let $V_1$ be the set of the first $\delta n /(4\omega')$ vertices of $P_0'$. Recall that we omit floors and ceilings,
and in particular we assume that $\delta n/(4(r-1)\omega')\in\NN$. Let $I_1$ be defined similarly for the following $\delta n/(4\omega')$ vertices
and $I_2$ for the last $\delta n/(4\omega')$ vertices of $P_0'$. Let $P_0$ be the loose path which results after deleting
$V_1$, $I_1$, $I_2$ and all incident edges.
Denote by $V_0$ the set of vertices contained in $P_0$.

Now let $\mathcal{A}$ be the set of $r$-sets such that one vertex lies in $I_1$, one vertex lies in $I_2$ and $r-2$ vertices lie in $V_1$.
If some $x\in\mathcal{A}$ forms an edge in $H_0$, then we would obtain a loose cycle in $H_2$
containing $P_0$ and thus of length at least
$L_P - 3n/(4(r-1)\omega') = L_P-o(n)$, as required.
We have
\[
|\mathcal{A}|=|I_1||I_2|\binom{|V_1|}{r-2}
=\left(\frac{n}{\omega'}\right)^2\binom{n/(4\omega')}{r-2}
=\Theta\left(\left(\frac{n}{\omega'}\right)^r\right).
\]
Therefore the probability that no $x\in\mathcal{A}$ is an edge in $H_0\sim H^r(n,p_0)$ is at most 
\begin{align*}
\left(1-p_0\right)^{\Theta((n/\omega')^r)} &\leq \exp\left(-p_0\Theta((n/\omega')^r)\right) 
\leq\exp\left(-\frac{p_1}{\omega}\Theta\left(\frac{n^r}{(\omega')^r}\right)\right) =o(1),
\end{align*}
and therefore with high probability there is at least one such edge, as required.
\end{proof}

\subsection{Offspring distribution}\label{app:offspringdist}

\begin{proof}[Proof of Proposition~\ref{prop:distributionofnochildren}]
We will consider an auxiliary breadth-first search algorithm starting from $w$, which goes through all
$\binom{n-1}{r-1}$ many $r$-sets of variable nodes containing the currently active node $x$
and queries an $r$-set $R$ to determine whether it is an edge of $H^r(n,p)$ (i.e.\ if there exists a factor node
whose neighbourhood is this $r$-set) if $x$ is the only node of $R$ which lies in the current tree,
and otherwise the query is skipped.
Initially the current tree consists simply of $w$, and the tree is updated if a query turns out to be an edge---in this case
the corresponding factor node and all its neighbours are added to the current tree and we proceed.
The event $\eone$ implies that $|D_{\leq\ell+1}(w)|\leq(\log n)^2$ and thus the number $q_x$ of queries made from any node $x$ satisfies
$\binom{n-(\log n)^2}{r-1}\leq q_x \leq \binom{n}{r-1}$.
The number of edges found is asymptotically distributed as $\Bi\left(q_x,p\right)$---this is not exact
since we still have the conditioning on the event $\eone$, but since by Lemma~\ref{lem:eventE} this is a very likely event,
it does not affect the distribution significantly.
Now we can couple the number $Z_x$ of edges discovered from $x$ from below and above by
random variables with distributions $\Bi\left(\binom{n-(\log n)^2}{r-1},p\right)$ and $\Bi\left(\binom{n}{r-1},p\right)$, respectively.
Since each of these random variables tends asymptotically to the $\Po(\binom{n}{r-1}p)$ distribution,
which in turn converges to the $\Po(d)$ distribution, so does $Z_x$.

Finally, to prove independence observe that the upper and lower couplings in the previous step
were independent of the number of edges found from any previous vertex (the lower coupling
only used the conditioning on $\eone$). Furthermore no $r$-set is queried more than once, and therefore
the upper and lower couplings can be considered independent for each vertex.
Thus the $Z_x$ are also asymptotically independent of each other.
\end{proof}

\subsection{Maximum degree}\label{app:maxdeg}
We aim to prove Claim~\ref{claim:maxdeg}, a statement about the maximum degree of
$G^r(n,p) = G(H^r(n,p))$. We first prove the analogous statement for $H^r(n,p)$.
For any $r$-uniform hypergraph $H$, let $\Delta(H)\coloneqq\max\limits_{v\in V(H)}d_H(v)$.
\begin{lemma}\label{lem:maxdegree}
Let $r\in\NN_{\geq 2}, \ d\in\mathbb{R}^+\backslash\{d^*\}$ and $p=\frac{d_n}{\binom{n-1}{r-1}}, \ d_n\to d$.
Then, for sufficiently large $n$, with probability at least $1-n^{-1/3}$ we have
\[
\Delta(\cH)\leq \log n.
\]
 \end{lemma}
 
\begin{proof}
By Lemma~\ref{lem:chernoffbounds} and taking the union bound we have
\begin{align*}
\mathbb{P}(\Delta(\cH)\geq \log n) &\leq n\cdot\mathbb{P}\left(\mathrm{Bin}\left(\binom{n-1}{r-1},p\right)>\log n\right) \\
&\leq 2n\cdot\exp\left(-\frac{\left(\log n-d_n\right)^2}{2\left(d_n+\frac{\log n}{3}\right)}\right) \\
&\leq 2\cdot\exp\left(\left(-\frac{3}{2}+o(1)\right)\log n+\log n\right)\\
&\le n^{-1/3}.\qedhere
\end{align*}
\end{proof}

\begin{proof}[Proof of Claim~\ref{claim:maxdeg}]
The statement follows immediately from the observation that 
\[\Delta(G^r(n,p))=\max\{\Delta(H^r(n,p)), \, r\}\]
and the fact that $r \le \log n$ for sufficiently large $n$.
\end{proof}

\subsection{Proof of Lemma~\ref{claim:terminationprob}}\label{sec:proofofterminationprob}
\begin{proof}
Defining $\varprop{j}(G')$ to be the proportion of variable nodes of degree $j$ in $G'$ for each $j\in \NN$,
observe that when revealing the neighbour of a factor node, the probability of revealing a variable node of degree at least three is
\begin{align*}
\frac{\sum_{j\ge 3}j\varprop{j}(G')}{\sum_{j \in \NN}j\varprop{j}(G')}
& \numgeq{\text{\phantom{L.\ref{lem:mainlemma1}}}} \frac{3\varprop{3}(G')}{\sum_{j =1}^3 j\varprop{j}(G')} \\
& \numgeq{\text{\phantom{L.\ref{lem:mainlemma1}}}} \frac{3(\varpropl{3}{\ell} - 2\eps_2)}{\sum_{j=1}^3 j(\varpropl{j}{\ell}+2\eps_2)} \\
& \numgeq{\text{L.\ref{lem:mainlemma1}}} \frac{3\Pr\left(\widetilde\Po(d\intfactsurvlim)=3\right)-3\eps-6\eps_2}{ \sum_{j=1}^3 j\Pr\left(\widetilde\Po(d\intfactsurvlim)=j\right)+6\eps+12\eps_2}\\
& \numgeq{\text{\phantom{L.\ref{lem:mainlemma1}}}} \frac{(d\intfactsurvlim)^3 \exp(-d\intfactsurvlim)/2}{\EE\left(\widetilde \Po(d\intfactsurvlim)\right)} - 9\eps-18\eps_2\\
& \numgeq{\text{\phantom{L.\ref{lem:mainlemma1}}}} \frac{(d\intfactsurvlim)^2 \exp(-d\intfactsurvlim)}{2} - 20\eps_2,
\end{align*}
where in the last line we used that $\eps_2=\sqrt{\eps}$.
Similarly, the probability of revealing a factor node of degree at least three is at least
\[
\frac{3\Pr\left(\widetilde\Bi(r,\intvarsurvlim)=3\right)-3\eps-6\eps_2}{ \sum_{j=1}^3 j\Pr\left(\widetilde\Bi(r,\intvarsurvlim)=j\right)+6\eps+12\eps_2}\geq \frac{(r-1)(r-2)\intvarsurvlim^2 (1-\intvarsurvlim)^{r-3}}{2} - 20\eps_2,
\]
which proves our claim.
\end{proof}

\section{Proof of Lemma~\ref{lem:eventE}}\label{app:eventE}

Recall that for a node $w\in\cV\cup\cF$ the event $\eone$ holds if and only if $|D_{\leq \ell+1}(w)|\leq(\log n)^2$
and furthermore $D_{\leq \ell+1}(w)$ contains no node which lies in a cycle of length at most~$2\ell$.
We will prove Lemma~\ref{lem:eventE} under the assumption that $w \in \cV$.
The case when $w \in \cF$ follows since clearly $E_w(\ell)$ is implied by
$E_{v_1}(\ell) \cap \ldots \cap E_{v_r}(\ell)$, where $v_1,\ldots,v_r$ are the neighbours of $w$.
Furthermore, applying the statement of the lemma to the variable nodes $v_1,\ldots,v_r$ we have
\begin{align*}
\Pr(E_{v_1}(\ell) \cap \ldots \cap E_{v_r}(\ell)) & \ge 1- \sum_{i=1}^r (1-\Pr(E_{v_i}(\ell)))\\
& \ge 1- r\exp\left(-\Theta\left(\sqrt{\log n}\right)\right)\\
& = 1-\exp\left(-\Theta\left(\sqrt{\log n}\right)\right),
\end{align*}
as required. We therefore assume in the following that $w \in \cV$.

We begin a breadth-first search from $w$, and index the generations by a \emph{time} $t$.
We run this process until a stopping time $\Tstop$ which is defined as follows.
\begin{definition}
Let $w\in\cV$ be given and $\ell=\ell(n)=o(\log\log n).$ Let $\Tstop$ be defined as the smallest time $t$ such that one of the following stopping conditions is invoked:
\begin{enumerate}[label=\textnormal{\textbf{(S\arabic*)}}]
	    \item\label{stopsize} The BFS tree has size $(\log n)^2$;
	    \item\label{stopcycle} The BFS tree contains a node which lies in a cycle of length at most $2\ell$;
	    \item\label{stoptime} $t=\ell+1$.
\end{enumerate}
\end{definition}
We aim to show that \whp~\ref{stoptime} is applied first. 

\begin{proposition}\label{prop:stopsize}
With probability at least $1-\exp(-\Theta(\sqrt{\log n}))$, \ref{stopsize} is not applied.
\end{proposition}

\begin{proof}
We consider a branching process starting at one node and in which each variable node
has offspring distribution $\Bi\left(\binom{n-1}{r-1},p\right)$ independently,
and each factor node has $r-1$ offspring independently---this is an upper coupling on the BFS process started at~$w$.
(Note in particular that each factor node has precisely $r-1$ children since the parent
has already been counted and the root, which is the only node which has no parent, is a variable node.)

Let us consider the first generation $t_0$ in which the number of nodes is at least $\sqrt{\log n}$.
If no such $t_0$ exists, or if $t_0 \ge \ell+1$, then the total size of the first $(\ell+2)$ generations
is at most $(\ell+2)\sqrt{\log n} \le (\log n)^2$ as required.
On the other hand, if $t_0\le \ell+1$, let us set $x_t$ to be the size of the $t$-th generation, for each $t \in \NN$.
We will assume for technical convenience that $x_t \ge \sqrt{\log n}$ for all $t\ge t_0$, i.e.\ the size
of a generation does not decrease back below $\sqrt{\log n}$---if this does occur, we simply add in some
fictitious nodes to reach this size threshold, which is clearly permissible for an upper bound.

By Lemma~\ref{lem:maxdegree}, with probability at least $1-n^{-1/3}$ we have
$x_{t_0}\le (\log n)^{3/2}$.
Furthermore, by a Chernoff bound, for each even $t \ge t_0$ (meaning that the $t$-th generation
consists of variable nodes), the probability that
$x_{t+1} \ge 2\binom{n-1}{r-1}px_t$ is at most $\exp\left(-\Theta(\sqrt{\log n})\right)$.
Taking a union bound over all at most $\lceil\frac{\ell+3-t_0}{2}\rceil \le \ell+1$ even generations,
and recalling that each variable node has $r-1$ children,
we deduce that \whp\ the total size of the process up to generation $\ell+1$ is at most
\begin{align*}
t_0\sqrt{\log n} + (1+(r-1))(\log n)^{3/2}\sum_{i=0}^{\lceil\frac{\ell+3-t_0}{2}\rceil} \left(2\binom{n-1}{r-1} p\right)^{i}
& \le \log n + r(\log n)^{3/2} \sum_{i=0}^{\ell+1} \left(3d\right)^i\\
& = (\log n)^{3/2}\Theta\left((3d)^{\ell+1}\right) \\
& \le (\log n)^2,
\end{align*}
where the last line follows since $\ell = o(\log \log n)$.
In total, the error probability is at most $n^{-1/3} + \exp(-\Theta(\sqrt{\log n})) = \exp(-\Theta(\sqrt{\log n}))$.
\end{proof}

\begin{proposition}\label{prop:stopcycle}
With probability at least $1-2n^{-1/3}$, \ref{stopcycle} is not applied.
\end{proposition}
\begin{proof}
Let $X_{\leq 2\ell}$ be the random variable counting the number of variable nodes in $D_{\leq\ell+1}(w)$
which lie on cycles of length at most $2\ell$ in $\randfact$. Then
\[
\expec(X_{\leq 2\ell})\leq\sum_{i=2}^{2\ell} i \frac{n^i}{2i}\binom{n}{r-2}^ip^i \le \sum_{i=3}^{2\ell} (n^{r-1}p)^i=o(\log n),
\]
where the last line follows since $\ell = o(\log n)$ and $n^{r-1}p = O(1)$.
By Markov's inequality, with probability at least $1-n^{-1/3}$, at most $n^{1/3}\log n$ variable nodes lie on a cycle of length at most $2\ell$.
Assuming that this is indeed the case,
since \ref{stopsize} was not invoked we infer that the probability that there exists a node in $D_{\leq{\ell+1}}(w)$ that lies in a cycle of length at most $2\ell$ is at most $\frac{(\log n)^3 n^{1/3}}{n} \le n^{-1/3}$, and so with probability at least $1-2n^{-1/3}$ \ref{stopcycle} is not invoked first.
\end{proof}
\begin{proof}[Proof of Lemma~\ref{lem:eventE}]
By Propositions~\ref{prop:stopsize} and~\ref{prop:stopcycle}, the probability that \ref{stoptime} is invoked first
for some particular node $w$
is at least $1-\exp(-\Theta(\sqrt{\log n})) - 2n^{-1/3} \ge 1- \exp(-\Theta(\sqrt{\log n}))$.
In particular, since $\exp(-\Theta(\sqrt{\log n}))=o(1)$, by Markov's inequality \whp\ $(1-o(1))n$ nodes satisfy $\eone$.
\end{proof}

\section{Proof of Lemma~\ref{claim:uniformity}}\label{sec:proofofuniformity}

\begin{proof}
First observe that any permutation of $\cV$ does not affect the probability that
$G_\ell$ is equal to a factor graph $H$ with variable node set $\cV$,
and therefore we assume without loss of generality that $\phi \vert_\cV$ is the identity map.

We will further simplify the proof by simply identifying each $a \in \cF_1$ with
$\phi(a) \in \cF_2$. Note that this is an abuse of terminology: if it were in fact true
that $\phi(a)=a$, then the two factor nodes would represent exactly the same edge
in the original $r$-uniform hypergraph, and therefore have exactly the same neighbours.
However, although we identify the two factor nodes with one another, we do not
carry over these restrictions (indeed, otherwise we would necessarily have $H_1=H_2$).
An alternative way of considering this is to say that we regard the factor nodes
no longer as edges of a hypergraph but as abstract nodes stripped of all information.

Now for $i=1,2$, let $\cG_i$ be the set of factor graphs $G$ of $r$-uniform hypergraphs
such that $G_\ell = H_i$.
Observe that any $G \in \cG_i$ has precisely the same node set as $H_i$, and in particular
has $|\cF_i|$ factor nodes, and therefore the probability that $G^r(n,p)=G$ is simply
$p^{|\cF_i|}(1-p)^{\binom{n}{r}-|\cF_i|}$. Since this value is identical for $i=1,2$,
what remains to prove is simply that
$$
|\cG_1|=|\cG_2|.
$$

We now observe that, as follows from the definition of the peeling process,
if $G_\ell = H_i$ then any edge of $G$ which runs between nodes which are
non-isolated in $H_i$ is also in $H_i$. (In other words, the edge set of $H_i$
is induced by the set of non-isolated nodes.) Since the sets of isolated nodes
in $H_1,H_2$ are identical, let $\cW$ be this set, so for any $G \in \cG_i$,
the subgraph of $G$ with all nodes but only those edges of $G$ which lie within $\cW$ is precisely $G_\ell=H_i$.
Let us define a
graph function $f:\cG_1 \to \cG_2$ which, for each $G \in \cG_1$,
changes the edges within $\cW$ to those of $H_2$ instead of $H_1$, but otherwise
leaves all nodes and edges unchanged. We aim to prove that $f$
is a bijection from $\cG_1$ to $\cG_2$, which implies that these classes
have the same size, as required.

The critical part of the proof is to observe that the range of this function
does indeed lie within $\cG_2$, i.e.\ that for every $G \in \cG_1$
we have $f(G) \in \cG_2$.
To see this, observe that it is a simple exercise to prove
by induction on $t$ that $f(G_t)=(f(G))_t$
for all $G \in \hat \cG_1$ and all $t \in [\ell]_0$, and in particular
$H_2 = f(H_1) = f(G_\ell) = (f(G))_\ell$.

However, slightly more subtly, we have to observe that $f(G)$ is indeed
a factor graph arising from an $r$-uniform hypergraph. It is clear from the
construction that $f(G)$ respects the bipartition of variable and factor nodes,
and also that the degrees of all nodes are identical in $G$ and $f(G)$, so in particular
every factor node has degree $r$ in $f(G)$.
Note also that $G$ contained no double edges which means that $f(G)$ cannot
contain double edges which do not lie entirely within $\cW$,
while the fact that $H_2$ contains no double edges also means that
$f(G)$ contains no double edges which lie entirely within $\cW$.
It therefore remains to prove that $f(G)$ contains no $r$-duplicates.
Observe that in an $r$-duplicate already all
nodes have degree at least two, and therefore none of these nodes will ever
be disabled in the peeling process. Thus if $f(G)$ contains an $r$-duplicate,
it is also present in $f(G)_\ell = H_2$, which contradicts our assumption.

We now further observe that $f$
is clearly an injection, since any two distinct graphs $G,G' \in \cG_1$
must differ in edges which do not lie completely within $\cW$, and therefore
$f(G),f(G')$ also differ in those edges. Finally, the function $f$ has an obvious inverse,
and therefore is a bijection. It follows that $|\hat \cG_1| = |\hat \cG_2|$.
\end{proof}

\end{document}